\documentclass[12pt]{article}
\usepackage{amssymb}
\usepackage{amsmath}
\usepackage{amsthm}
\usepackage{amsfonts}
\usepackage{color}
\usepackage{enumerate}
\usepackage{hyperref}
\usepackage{graphicx}
\usepackage{authblk}

\def\R{\mathbb{R}}
\def\N{\mathbb{N}}
\def\ns{\mathcal H (\Omega)}
\def\calh{\mathcal H}
\def\fa{\hbox{ for all }}

\DeclareMathOperator*{\argmax}{arg\,max}

\newtheorem{thm}{Theorem}
\newtheorem{prop}[thm]{Proposition}
\newtheorem{cor}[thm]{Corollary}
\newtheorem{lemma}[thm]{Lemma}
\newtheorem{definition}[thm]{Definition} 
\theoremstyle{remark}
\newtheorem{rem}[thm]{Remark}

\providecommand{\keywords}[1]{\noindent\textbf{\textit{Keywords: }} #1}
\providecommand{\MSC}[1]{\textbf{\textit{2000 MSC: }} #1}

\title{Approximation of Eigenfunctions in Kernel-based Spaces}
\author[1]{Gabriele Santin\thanks{gsantin@math.unipd.it}}
\author[2]{Robert Schaback\thanks{schaback@math.uni-goettingen.de}}
\affil[1]{Dipartimento di Matematica, University of Padova}
\affil[2]{Institut f\"ur Numerische und Angewandte Mathematik, Universit\"at G\"ottingen}
\date{\today}

\begin{document}
\maketitle
\begin{abstract} Kernel-based methods in Numerical Analysis have the advantage of 
yielding optimal recovery processes in the ''native'' Hilbert space $\calh$ 
in which they are reproducing. Continuous kernels on compact domains 
have an expansion into eigenfunctions that are both $L_2$-orthonormal 
and orthogonal in $\calh$ (Mercer expansion). 
This paper examines the corresponding eigenspaces and 
proves that they have optimality properties among all other 
subspaces of $\calh$. These results have strong connections 
to $n$-widths in Approximation Theory, and they establish 
that errors of optimal approximations are closely related to the decay of the eigenvalues.

Though the eigenspaces and eigenvalues are not readily available, 
they can be well approximated using the standard $n$-dimensional 
subspaces spanned by translates of the kernel with respect to $n$ 
nodes or centers. We give error bounds for the numerical 
approximation of the eigensystem via such subspaces. 
A series of examples shows that our numerical technique 
via a greedy point selection strategy allows to calculate the 
eigensystems with good accuracy.
\end{abstract}

\\\MSC{41Axx,  41A46,  41A58,  42Cxx,  42C15,  42A82,  45P05,  46E22,  47A70, 65D05,  65F15,  65R10, 65T99}


\section{Introduction}
We start with a few background facts about kernel-based methods. Details
can be retrieved from the monographs
\cite{buhmann:2003-1,We2005,fasshauer:2007-1}
and the surveys \cite{buhmann:2000-1,schaback-wendland:2006-1}.
Let $\Omega\subset\R^d$ be a nonempty set, and let $K: \Omega\times \Omega\to\R$ 
be a positive definite and symmetric kernel on $\Omega$. Associated with $K$ 
there is a unique \textit{native space} $\ns$, that is a separable Hilbert space of 
functions $f:\Omega\to\R$ where K is the reproducing kernel. This means that 
$K(\cdot, x)$ is the Riesz representer of the evaluation functional $\delta_x$, i.e.,   
\begin{equation}\label{repProp}
f(x) = (f, K(\cdot, x)),\;\; \fa x\in \Omega,\;f\in \ns
\end{equation}
where we use the notation $(\cdot, \cdot)$, without subscript, to denote here
and 
in the following the inner product of $\ns$. Also the converse holds: any
Hilbert 
space on $\Omega$ where the evaluation functionals $\delta_x$ are continuous for 
any $x\in\Omega$ is the native space of a unique kernel.
 
One way to construct the native space is as follows. First one considers the 
space 
$\calh_0(\Omega) = \text{span}\{K(\cdot,x), x\in\Omega\}$ and then 
equips
it with 
the positive definite and symmetric bilinear form
\begin{equation*}
\left(\sum_j c_j K(\cdot, x_j),\sum_i d_i K(\cdot, x_i)\right):= \sum_{j,i} c_j d_i K(x_j, x_i).
\end{equation*} 
The native space $\ns$ then is the closure of 
$\calh_0(\Omega)$ with respect to the inner product defined by this form.

Given a finite set $X_n = \{x_1,\dots,x_n\}\subset\Omega$ of distinct points,
the interpolant $s_{f,X_n}$ of a function $f\in\ns$ on $X_n$ is 
uniquely defined, since the \textit{kernel matrix}
\begin{equation*}
A = [K(x_i,x_j)]_{i,j=1}^n
\end{equation*}
is positive definite, the kernel being positive definite. 
Letting $V(X_n)\subset\ns$ be the subspace spanned by the kernel translates 
$K(\cdot, x_1)$, $\dots$,  $K(\cdot, x_n)$, the interpolant $s_{f,X_n}$ is the projection of $f$ into $V(X_n)$ with respect to $(\cdot,\cdot)$. 
 
The usual way to study the approximation error is to consider the 
\textit{Power Function} $P_n(x)$ of $X_n$ at $x$, that is the $\ns$-norm of 
the pointwise error functional $f\mapsto f(x)-s_{f,X_n}(x)$ at $x$, but we will consider also other measurements of the interpolation error, and other projectors.

We make the additional 
assumptions that $\Omega$ is a compact set in $\R^d$ 
and the kernel is continuous on $\Omega\times\Omega$. This ensures that 
the native space has a continuous embedding into $L_2(\Omega)$. Indeed, 
using the reproducing property \eqref{repProp} we have
\begin{equation*}\label{embedding}
\|f\|_{L_2}\leq \left(\int_{\Omega}K(x,x)dx\right)^{1/2} \|f\|\;\; \fa f\in\ns 
\end{equation*}
where the integral of the kernel is finite. This allows to define a 
compact and self-adjoint integral operator $T:L_2(\Omega)\to L_2(\Omega)$ 
which will be crucial for our goals. For $f\in L_2(\Omega)$ we define 
\begin{equation}
T f(x) = \int_{\Omega} K(x,y) f(y) dy,\;\; x\in\Omega.
\end{equation}
It can be shown that the range $T(L_2(\Omega))$ of $T$ is dense in $\ns$, and
\begin{equation}\label{TandInn}
(f,g)_{L_2} = (f,T g) \text{ for all } f\in\ns,\; g\in L_2(\Omega). 
\end{equation}
\begin{rem}
The operator $T$ can be defined using any positive and finite measure $\mu$ 
with full support on $\Omega$ (see \cite{SuWu2009}) and the same properties 
still hold, but we will concentrate here on the Lebesgue measure.
\end{rem}

The following theorem (see e.g. \cite[Ch. 5]{Pogo66}) 
applies to our situation, and provides a way to represent 
the kernel as an \textit{expansion} (or \textit{Hilbert - Schmidt} or 
\textit{Mercer}) kernel (see e.g. \cite{SchWen2002,schaback:1999-2}).
\begin{thm}[Mercer]\label{mercer}
If $K$ is a continuous and positive definite kernel on a compact set $\Omega$, the
operator $T$ has a countable set of positive eigenvalues 
$\lambda_1\geq\lambda_2\geq\dots> 0$ and eigenfunctions
$\{\varphi_j\}_{j\in\N}$ with $T\varphi_j = \lambda_j\varphi_j$. The
eigenfunctions are orthonormal in $L_2(\Omega)$ and orthogonal in $\ns$ with
$\|\varphi_j\| = \lambda_j^{-1}$. Moreover, the 
kernel can be decomposed as
\begin{equation}\label{eq:expansion}
K(x,y) = \sum_{j=1}^{\infty} \lambda_j\ \varphi_j(x)\ \varphi_j(y)\;\;\; x,y\in\Omega, 
\end{equation}
where the sum is absolutely and uniformly convergent.
\end{thm}
From now on we will call $\{\sqrt{\lambda_j} \varphi_j\}_{j\in\N}$ the 
\textit{eigenbasis}, and use the notation $E_n = \text{span}\{\sqrt{\lambda_j} \varphi_j, j=1,\dots,n\}$.

We recall that it is also possible to go the other way round and define a 
positive definite and continuous kernel starting from a given sequence of 
functions $\{\varphi_j\}_j$ and weights $\{\lambda_j\}_j$, provided some 
mild conditions of summability and linear independence. See \cite{SchWen2002} 
for a detailed discussion about this construction, and note that eigenfunction
expansions play a central role in the RBF-QR technique dealt with in various
papers 
\cite{
fornberg-et-al:2011-1, fasshauer-mccourt:2012-1} recently. Furthermore,
eigenexpansion techniques are a central tool when working with kernels on
spheres and Riemannian manifolds
\cite{dyn-et-al:1999-1,jetter-et-al:1999-1,narcowich-et-al:2002-1,fornberg-piret:2007-1}.

In this paper we shall study the eigenbasis in detail and compare the 
eigen\-spaces to other $n$-dimensional subspaces of $\ns$. 
General finite dimensional subspaces and their associated $L_2(\Omega)$- and
$\ns$- projectors are treated in Section \ref{sec:Pro}. In particular, the
approximation error is bounded in terms of generalized Power Functions  
that turn out to be very useful for the rest of the paper. The determination of
error-optimal 
$n$-dimensional subspaces is the core of Section \ref{sec:MinErrSub}, 
and is treated there by $n$-widths, proving that eigenspaces are 
optimal under various circumstances. In addition to the case of 
Kolmogorov $n$-width as treated in \cite{SchWen2002}, we prove that 
eigenspaces minimize 
the $L_2(\Omega)$ norm of the Power Function.

In Section \ref{sec:restriction} we move towards numerical 
calculation of eigenbases by focusing on (possibly finite-dimensional) 
closed subspaces. In particular, we want to use subspaces $V(X_n)$ 
spanned by kernel translates $K(\cdot,x_1)$, $\dots$, $K(\cdot,x_n)$ 
for point sets $X_n$ to calculate approximations  
to the eigenbasis.  By means of Power Functions we can bound the error 
committed by working with finite dimensional subspaces.

Section \ref{sec:asympEig} focuses on the decay of eigenvalues. 
We recall the fact that increased smoothness of the kernel leads 
to faster decay of eigenvalues. We prove that using point-based 
subspaces $V(X_n)$ we can approximate the 
first $n$ eigenvalues with an error connected to the decay of $\lambda_n$. 

Section \ref{sec:alg} describes the numerical algorithms 
used for the examples in Section \ref{sec:num}. In particular, 
we use a greedy method for selecting sets $X_n$ of $n$ points 
out of $N$ given points such that eigenvalue calculations in $V(X_n)$ are stable and efficient.

Finally, Section \ref{sec:num} shows that our algorithm allows 
to approximate the eigenvalues for Sobolev spaces in a way that 
recovers the true decay rates, and by results of Section 
\ref{sec:asympEig} we have 
bounds on the committed error. 
For Brownian bridge kernels the eigenvalues are known, 
and our algorithm approximates  
them very satisfactorily.


\section{Projectors}\label{sec:Pro}
As mentioned in the introduction,  
the interpolation problem in $\ns$ is well
defined, and the interpolation operator is a $\ns$-projector into the spaces generated
by translates of the kernel. But we want to look also at fully general 
subspaces $V_n$ of $\ns$, generated by any set of $n$ basis functions, and at
other 
linear approximation processes defined on such spaces,
e.g. approximations in the $L_2(\Omega)$ norm.

For instance, we consider the two projectors
$$
\begin{array}{rcll}
\Pi_{L_2,V_n} f &=& \sum_{j=1}^n(f, w_j)_{L_2} w_j,&f\in\ns, \\
\Pi_{\calh,V_n} f &=& \sum_{j=1}^n(f, v_j) v_j,&f\in\ns, 
\end{array}
$$
where the $w_j$ and $v_j$ are $L_2(\Omega)$- and $\ns$- orthonormal basis 
functions of $V_n$,
respectively. The first projector is defined on all of $L_2(\Omega)$ and can be 
analyzed on all intermediate spaces. We want to see how these projectors are connected. 

The two projectors do not coincide in general, but there is a 
special case. For the sake of clarity we present here the proof 
of the following Lemma, even if it relies on a result that is proven in Section \ref{sec:restriction}.
\begin{lemma}\label{TheVn}
If the projectors coincide on $\ns$ for an $n$-dimensional 
space $V_n$, then $V_n$ is spanned by $n$ eigenfunctions. The converse also holds.
\end{lemma}
\begin{proof}
We start with the converse. For each fixed $\varphi_j$, thanks to \eqref{TandInn}, we have
\begin{equation*}
(f,\varphi_j)_{L_2} = {\lambda_j} (f,\varphi_j) \text{ for all } f\in \ns.
\end{equation*}
Then 
\begin{align*}
\Pi_{L_2,V_n} f = & \sum_{j=1}^n(f,\varphi_j)_{L_2}\varphi_j  = \sum_{j=1}^n{\lambda_j}
(f,\varphi_j) \varphi_j   \\ 
= & \sum_{j=1}^n (f,\sqrt{\lambda_j}\varphi_j) \sqrt{\lambda_j}\varphi_j = \Pi_{\calh,V_n} f. 
\end{align*}
Assume now that the projectors coincide on $\ns$. We can choose a basis $\{v_j\}_j$ 
of $V_n$ which is $L_2(\Omega)$-orthonormal and $\ns$-orthogonal (see Lemma \ref{doubleOrth}), 
with $\|v_j\|^2 = 1 / \sigma_j$. Since $\Pi_{L_2, V_n}f = \Pi_{\calh, V_n}f$ for any $f\in\ns$, necessarily 
\begin{equation*}
(f,v_j)_{L_2} = {\sigma_j} (f,v_j) \text{ for all } f\in \ns\text{ and } j = 1, \dots, n,
\end{equation*}
and in particular for $f = K(\cdot, x)$, $x\in\Omega$. Consequently, $\{v_j\}_j$ and 
$\{\sigma_j\}_j$ are eigenfunctions / eigenvalues of $T$.
\end{proof}

We are now interested in an error analysis of the approximation 
by functions from these general subspaces, and we want to allow both of the above projectors.
\begin{definition}\label{DefGenPow}
For a normed linear space $H$ of functions on $\Omega$ and a linear operator 
$\Pi$ on $H$ such that all the functionals $\delta_x - \delta_x\circ \Pi$ are 
continuous for some norm $\|\cdot\|_H$, the \textit{generalized Power Function} in $x\in\Omega$ is the norm of the error functional at $x$, i.e.,  
\begin{equation}\label{genPF}
P_{\Pi,\|.\|_H}(x):= \sup_{\|f\|_H\leq 1} |f(x)-(\Pi f)(x)|. 
\end{equation}
\end{definition} 
The definition fits our situation, because we are free to 
take $\Pi=\Pi_{\calh,V_n}$ with $\|.\|_H=\|.\|$, the normed linear space $H$ being $\ns$.

In the following, when no confusion is possible, we will
use the simplified notation $P_{V_n,\calh}$ or just $P_{V_n}$ to denote the 
Power Function of $\Pi_{\calh,V_n}$ with respect to $\|\cdot\|$. 

To look at the relation between generalized Power Functions, subspaces and bases 
we start with the following lemma.

\begin{lemma}\label{LemSumEx}
If a separable Hilbert space $H$ of functions on $\Omega$ has 
continuous point evaluation, then each $H$-orthonormal basis $\{v_j\}_j$ satisfies 
\begin{equation*}
\sum_j v_j^2(x) <\infty.
\end{equation*}
Conversely, the above condition ensures that all point evaluation functionals are continuous.
\end{lemma}
\begin{proof}
The formula
\begin{equation*}
f = \displaystyle{\sum_j(f,v_j)_H v_j} 
\end{equation*}
holds in the sense of limits in $H$. 
If point evaluations are continuous, we can write
$$
f(x)=\displaystyle{\sum_j(f,v_j)_H v_j(x)\text{ for all } x\in\Omega} 
$$
in the sense of limits in $\R$. Since the sequence $\{(f,v_j)_H\}_j$
is in $\ell_2$ and can be an arbitrary element of that space, the sequence
$\{v_j(x)\}_j$ must be in $\ell_2$, because the above expression is a continuous
linear functional on $\ell_2$. 
 
For the converse, observe that for any $n\in\N$, $x\in\Omega$ and $\sum_{j=1}^n c_j^2\leq 1$ the term $\sum_{j=1}^n c_j v_j(x)$ is bounded above by $\sum_j v_j^2(x)$, which is finite for any $x\in\Omega$. Hence, for any $x\in\Omega$, $\sup_{\|f\|_H\leq 1} |f(x)|$ is uniformly bounded for $f\in H$.
\end{proof}
\begin{lemma}\label{LemProjPow}
For projectors $\Pi_{V_n}$ within separable Hilbert spaces $H$ of functions on
some 
domain $\Omega$ onto finite-dimensional subspaces $V_n$ generated by 
{$H$}-orthonormal 
functions $v_1,\ldots,v_n$ that are completed, we have
\begin{equation*}
P^2_{\Pi_{V_n},\|.\|_H}(x)= \sum_{k>n}v_k^2(x)
\end{equation*}
provided that all point evaluation functionals are continuous.
\end{lemma} 
\begin{proof}
The pointwise error at $x$ is
\begin{align*}
f(x)-\Pi_{V_n}f(x)
&= 
\sum_{k>n}(f,v_k)_Hv_k(x),
\end{align*}
and, thanks to the previous Lemma, we can safely bound its norm as 
\begin{align*}
|f(x)-\Pi_{V_n}f(x)|^2
&\leq 
\sum_{k>n}(f,v_k)_H^2\sum_{j>n}v_j^2(x)\\
&= 
\|f-\Pi_{V_n}f\|_H^2\sum_{j>n}v_j^2(x)\leq  
\|f\|_H^2\sum_{j>n}v_j^2(x),
\end{align*}
with equality if $f\in V_n^\perp$.
\end{proof}
This framework includes also the usual Power Function. 
\begin{lemma}\label{LemPointProj}
Let $X_n$ be a set of $n$ points in a compact domain $\Omega$, and let $V(X_n)$ be 
spanned by the $X_n$-translates of $K$. Then the above notion of
$P_{V(X_n)}$ coincides with the standard notion of the interpolatory Power Function wrt. $X_n$.  
\end{lemma}
The proof follows from the fact that the interpolation operator coincides 
with the projector $\Pi_{\calh,V(X_n)}$ and the Power Function is in both cases 
defined as the norm of the pointwise error functional.

\section{Minimal error subspaces}\label{sec:MinErrSub}
In this section we present the main results of this paper. 
Our goal is to analyze the behavior of the approximation error, 
considered from different points of view, depending on the $n$-dimensional 
subspace $V_n$. Roughly speaking, we will see that it is possible to 
exactly characterize the $n$-dimensional subspaces of minimal error, both considering 
the $L_2(\Omega)$ norm of the error and the the $L_2(\Omega)$ norm of the pointwise error.  

The way this problem is addressed in Approximation Theory is the study 
of \textit{widths} (see e.g. the comprehensive monograph \cite{Pin1985}, 
and in particular Chapter 4 for the theory in Hilbert spaces).
 
We will concentrate first on the $n$-width of Kolmogorov. The 
Kolmogorov $n$-width $d_n(A;H)$ of a subset $A$ in an Hilbert space $H$ is defined as
\begin{equation*}
d_n(A;H) := \inf_{\stackrel{V_n \subset H}{\dim(V_n)=n}}\;\;\;\sup_{f \in A}\;\;\;\inf_{v\in V_n} \|f - v_n\|_{H}.
\end{equation*}
It measures how $n$-dimensional subspaces of $H$ can approximate a given subset
$A$. 
If the infimum is attained by a subspace, this is called an  
\textit{optimal subspace}. The interest is in characterizing optimal 
subspaces and to compute or estimate the asymptotic behavior of the 
width, usually letting $A$ to be the unit ball $S(H)$ of $H$.

The first result that introduces and studies  
$n$-widths for native spaces was presented in \cite{SchWen2002}. 
The authors consider the $n$-width 
$d_n(S(\ns); L_2(\Omega))$, simply $d_n$ in the following, and prove that 
\begin{equation*}
d_n = \inf_{\stackrel{V_n \subset L_2}{\dim(V_n)=n}}\;\;\;
\sup_{f \in S(\calh)} \|f - \Pi_{L_2,V_n} f\|_{L_2} = \sqrt{\lambda_{n+1}},
\end{equation*}
and the unique optimal subspace is $E_n$. 
 
This result is the first that exactly highlights the importance of 
analyzing the expansion of the operator $T$ to better understand the 
process of approximation in $\ns$. In the following we will try to deepen this
connection.
Our main concern will be to replace the $L_2(\Omega)$ projector
$\Pi_{L_2,V_n}$ by the $\ns$ projector $\Pi_{\calh,V_n}$, while still keeping the
$L_2(\Omega)$ norm to measure the error. 
The $\ns$ projector is closer to the standard interpolation projector,
and it differs from the $L_2(\Omega)$ projector unless the space $V_n$ is an eigenspace,
see Lemma \ref{TheVn}.

We consider the $L_2(\Omega)$ norm of the error functional in $\ns$ for the 
projection $\Pi_{\calh,V_n}$ into a subspace $V_n\subset\ns$, i.e.,
\begin{equation*}
\sup_{\|f\|_{\calh}\leq  1}\|f-\Pi_{\calh,V_n}f\|_{L_2}
\end{equation*}
and we look for the subspace which minimizes this quantity. 
In other words, following the definition of the Kolmogorov $n$-width, we can define  
\begin{equation*}
\kappa_n:=
\inf_{\stackrel{V_n \subset \calh}{\dim(V_n)=n}}\;\;\;\sup_{f\in S(\calh)}\|f-\Pi_{\calh,V_n}f\|_{L_2}.
\end{equation*}

We recall 
in the next Theorem that $\kappa_n$ is equivalent to $d_n$, i.e., 
the best approximation in $L_2(\Omega)$ of $S(\ns)$ with respect to
$\|\cdot\|_{L_2}$ 
can be achieved using $\ns$ itself and the projector $\Pi_{\calh,V_n}$. 
The result can be found in \cite{NoWo2008}. 

\begin{thm}\label{TheKolWidth}
For any $n>0$ we have
\begin{equation*}
\kappa_n = \sqrt{\lambda_{n+1}}, 
\end{equation*}
and the unique optimal subspace is $E_n$.
\end{thm} 
\begin{proof}
Since $\ns \subset L_2(\Omega)$ and since $\Pi_{L_2, V_n} f$ is the 
best approximation from $V_n$ of $f\in\ns$ wrt. $\|\cdot\|_{L_2}$, we have
\begin{align*}
d_n& = \inf_{\stackrel{V_n \subset L_2}{\dim(V_n)=n}}\;\;\;\sup_{f\in S(\calh)}\|f-\Pi_{L_2,V_n}f\|_{L_2}\\
  &\leqslant  \inf_{\stackrel{V_n \subset \calh}{\dim(V_n)=n}}\;\;\;\sup_{f\in S(\calh)}\|f-\Pi_{L_2,V_n}f\|_{L_2}\\
  &\leqslant \inf_{\stackrel{V_n \subset \calh}{\dim(V_n)=n}}\;\;\;\sup_{f\in S(\calh)}\|f-\Pi_{\calh,V_n}f\|_{L_2}= \kappa_n.
\end{align*}
On the other hand, since $\Pi_{L_2,E_n} = \Pi_{\calh,E_n}$ on $\ns$ (Lemma \ref{TheVn}), 
\begin{align*}
\kappa_n & \leqslant \sup_{f\in S(\calh)}\|f-\Pi_{\calh,E_n}f\|_{L_2}\\
  & = \sup_{f\in S(\calh)}\|f-\Pi_{L_2,E_n}f\|_{L_2} = d_n,	
\end{align*}
since $E_n$ is optimal for $d_n$. Hence $\kappa_n = d_n = \sqrt{\lambda_{n+1}}$.
\end{proof}

We now move to another way of studying the approximation error. 
Instead of directly considering the
$L_2(\Omega)$ norm of approximants, we first take the 
norm of the pointwise error of the  $\Pi_{\calh,E_n}$ projector  
and then minimize its $L_2(\Omega)$ 
norm over $\Omega$. This means to find a subspace which minimizes
the $L_2(\Omega)$ norm $\|P_{V_n}\|_{L_2}$  of the Power Function $P_{V_n}$  
among all $n$-dimensional subspaces $V_n\subset\ns$. Using the definition of the 
generalized Power Function, we can rephrase the problem in the fashion of the previous results by defining 
\begin{equation*}
p_n := \inf_{\stackrel{V_n \subset \calh}{\dim(V_n)=n}}\;\;\;\left\|\sup_{f\in S(\calh)}|f(\cdot)-\Pi_{\calh,V_n}f(\cdot)|\right\|_{L_2}.
\end{equation*}

In the following Theorem, mimicking \cite[Theorem 1]{Ism1968}, we prove that, 
also in this case, the optimal $n$-dimensional subset is $E_n$, and $p_n$ 
can be expressed in terms of the eigenvalues. 
\begin{thm}\label{L2MinPow}
For any $n>0$ we have
\begin{equation*}
p_n  = \sqrt{\sum_{j>n} \lambda_j}, 
\end{equation*}
and the unique optimal subspace is $E_n$.
\end{thm}
\begin{proof}
For a subset $V_n$ we can consider a $\ns$-orthonormal basis $\{v_k\}_{k=1}^n$ and complete it to an orthonormal 
basis $\{v_k\}_{k\in\N}$ of $\ns$. We can move from the eigenbasis 
to this basis using a matrix $A = (a_{ij})$ as
\begin{equation}\label{MatChaBas}
v_k = \sum_{j=1}^{\infty} a_{jk} \sqrt{\lambda_j} \varphi_j,
\end{equation}
where $\sum_{j=1}^{\infty}a_{jk}^2=\sum_{k=1}^{\infty}a_{jk}^2=1$.
Hence, the power function of $V_n$ is
\begin{align*}
 P_{V_n}(x)^2 &= \sum_{k>n} v(x)^2 =\sum_{k>n} \left(\sum_{j=1}^{\infty} a_{jk} \sqrt{\lambda_j} \varphi_j(x)\right)^2 \\
&= \sum_{i,j=1}^{\infty} \sqrt{\lambda_i}\varphi_i(x) \sqrt{\lambda_j} \varphi_j(x) \sum_{k>n} a_{ik} a_{jk},
\end{align*}
and, defining $ q_j = \sum_{k=1}^n a_{jk}^2$, we can compute its norm as  
\begin{align*}
\|P_{V_n}\|_{L_2}^2 =&  \int_{\Omega} \sum_{k>n} \left(\sum_{j=1}^{\infty} a_{jk} \sqrt{\lambda_j} \varphi_j(x)\right)^2  dx  \\=& \sum_{k>n} \sum_{j=1}^{\infty} a_{jk}^2 \lambda_j = \sum_{j=1}^{\infty} \lambda_j - \sum_{j=1}^{\infty} q_j\lambda_j.
\end{align*} 
Now we need to prove that $\sum_{j=1}^{\infty} q_j\lambda_j\leqslant
\sum_{j=1}^{n} \lambda_j$. 

Let $m = \lceil \sum_j q_j\rceil\leq n$. We split the cumulative sum over the $q_j$ into integer ranges 
$$
i-1<\sum_{j=1}^{j_i} q_j\leqslant i,\;1\leqslant i\leq m.
$$
Then $j_m$ can be infinite, but $j_{m-1}$ is finite, and
since $0\leq q_j\leq 1$ we get stepwise
$$
\begin{array}{rcl}
0< \sum_{j=j_{i-1}+1}^{j_i} q_j&\leqslant& 1,\\
j_i-j_{i-1}
&\geq&  1,\\
j_i&\geq&  i,\\
j_i\geq j_{i-1}+1& \geq&  i\\
\end{array} 
$$
for $1\leq i\leq m$, using $j_0=0$.
Since the sequence of the eigenvalues is non negative and non increasing, 
this implies
\begin{eqnarray*}
\sum_{j=1}^{\infty} q_j\lambda_j 
&\leq & 
\sum_{i=1}^{m-1} q_{j_{i-1}+1}\lambda_{j_{i-1}+1}+ \lambda_{j_{m-1}+1}\sum_{j=j_{m-1}+1}^{j_m} q_j\\
&\leq & 
\sum_{i=1}^{m} \lambda_{j_{i-1}+1}\leq \sum_{i=1}^{m}\lambda_{i}\leq\sum_{i=1}^{n}\lambda_{i}.
\end{eqnarray*}

If we take $V_n = E_n$ and
$\{\sqrt{\lambda_j}\varphi_j\}_{j=1}^n$ as its basis, the matrix $A$ in
\eqref{MatChaBas} is the infinite identity matrix.  
Thus equality holds in the last inequality. 
\end{proof}

\section{Restriction to closed subspaces}\label{sec:restriction}
The previous results motivate the interest for the knowledge and 
study of the eigenbasis. But from a practical point of view 
there is some limitation: the eigenbasis cannot be computed in 
general,  and it can not be used for truly scattered data
approximation, 
since there exists at least a set $X_n\subset\Omega$ such that the 
collocation matrix of $E_n$ on $X_n$ is singular (by the Mairhuber-Curtis 
Theorem, see e.g. \cite[Theorem 2.3]{We2005}).

To overcome this problem we consider instead subspaces of $\ns$ of the 
form $V_N=V(X_N)= \text{span}\{K(\cdot,x):x\in X_N\}$, where $X_N = \{x_1,\dots,x_N\}$ 
is a possibly large but finite set of points in $\Omega$. 
The basic idea is to replace $\ns$ by $V(X_N)$ in order to get a good 
numerical approximation to the true eigenbasis with respect to $\ns$. 

To this end, we repeat the constructions of the previous section for a finite-dimensional native space, 
i.e., the problem of finding, for $n<N$, an $n$-dimensional 
subset $V_n$ which minimizes the error, in some norm, among all the subspaces of $V(X_N)$ of dimension $n$, and that can now be exactly computed. 

One could expect that the optimal subset for this restricted problem is the projection of $E_n$ into $V(X_n)$. In fact, as we will see, this is not the case, but the optimal subspace will still approximate the true eigenspaces in a near-optimal sense.

The analysis of such point-based subspaces can be carried out by 
looking at general closed subspaces of the native space.
It can be proven (see \cite[Th. 1]{MoSc2009}) that, if $V$ is a closed
subspace of $\ns$, it is the native space on $\Omega$ of the kernel $K_V(x,y) =
\Pi^x_{\calh,V}\Pi^y_{\calh,V} K(x,y)$, with inner product given by the restriction of
the one of $\ns$. 
The restricted operator $T_V:L_2(\Omega)\to L_2(\Omega)$ defined as
\begin{equation}\label{discreteTOperator}
T_V f(x) = \int_{\Omega} K_V(x,y) f(y) dy,\;\; x\in\Omega,
\end{equation}
maps $L_2(\Omega)$ into $V$. Then 
Theorem \ref{mercer} applied to this operator gives the eigenbasis for $V$ and $T_V$ on $\Omega$ and the corresponding
eigenvalues, which will be denoted as $\{\varphi_{j,V}\}_j$,
$\{\lambda_{j,V}\}_j$. They can be 
finitely or infinitely many, depending on the dimension of $V$. We will use the notation 
$E_{n,V} = \text{span}\{\sqrt{\lambda_{j,V}} \varphi_{j,V},\ j=1,\dots,n\}$ if $n \leq \dim(V)$.

This immediately proves the following Lemma that was already used in Section \ref{sec:Pro}. 
\begin{lemma}\label{doubleOrth}
Any closed subspace $V$ of the native space has a unique basis 
which is $\ns$-orthogonal and $L_2(\Omega)$-orthonormal.
\end{lemma}
Uniqueness is understood here like stating uniqueness 
of the eigenvalue expansion of the integral operator defined by the kernel,
i.e., the eigenspaces are unique.

Before analyzing the relation between the approximation in $V$ and in 
$\ns$ we establish a connection between the corresponding kernels and Power Functions.
\begin{lemma}
If $V\subset\ns$ is closed, 
\begin{equation}\label{powAndKers}
P_{V,\calh}(x)^2 = K(x,x) - K_V(x,x)\;\;\text{ for all } x\in\Omega.
\end{equation}
Moreover, if $U\subset V\subset\ns$ are closed, the Power Functions are related as
\begin{equation*}
P_{U,\calh}(x)^2 = P_{U,V}(x)^2 + P_{V,\calh}(x)^2\;\;\text{ for all } x\in\Omega.
\end{equation*}
\end{lemma}
\begin{proof}
The Power Function is the norm of the error functional. For $f\in\ns$, $\|f\|\leq 1$, and $x\in\Omega$ we have 
\begin{align*}
&|f(x) - \Pi_{\calh, V}f(x)| = |(f, K(\cdot,x) - K_V(\cdot, x)|\\
&\leq \|f\| \|K(\cdot,x) - K_V(\cdot, x)\| \leq \sqrt{K(x,x) - K_V(x, x)},
\end{align*}
with equality if $f$ is the normalized difference of the kernels. 
This proves \eqref{powAndKers}, and the relation between the Power Functions easily follows.
\end{proof}

Since $V$ is a native space itself, the results of the previous section 
hold also for $V$. We can then define in $V$ the analogous notions of $d_n$, $\kappa_n$
and $p_n$, and by Theorems \ref{TheKolWidth}, \ref{L2MinPow} we know that they are
all minimized by $E_{n,V}$, with values 
$\sqrt{\lambda_{n+1,V}}$, $\sqrt{\lambda_{n+1,V}}$,  
and $\sqrt{\sum_{j>n}\lambda_{j,V}}$, respectively. 

These results deal with the best approximation of the unit ball $S(V)$, but
allow also to face the problem of the constrained optimization in the case of
$p_n$, 
i.e., the minimization of the error of approximation of $S(\ns)$ using only subspaces of $V$.
Indeed, thanks to Lemma \ref{powAndKers}, we know that for any $V_n\subset V$ 
and for any $x\in\Omega$, the squared power functions of $\ns$ 
and of $V$ differ by an additive constant. This means that the 
minimality of $E_{n,V}$ does not change if we consider the standard power 
function on $\ns$. Moreover, 
\begin{align*}
\int_{\Omega}P_{E_n,\ns}^2(x) dx =& \int_{\Omega}P_{E_n,V}^2(x) dx + \int_{\Omega}P_{V,\ns}^2(x) dx \\
& = \sum_{j=1}^m \lambda_{j,V}- \sum_{j=1}^n \lambda_{j,V} + \sum_{j=1}^{\infty} \lambda_{j} - \sum_{j=1}^m \lambda_{j,V}\\
& = \sum_{j=1}^{\infty} \lambda_{j} - \sum_{j=1}^n \lambda_{j,V}.
\end{align*}

This proves the following corollary of Theorem \ref{L2MinPow}.
\begin{cor}\label{L2minVm}
Let $V\subset\ns$ be a closed subspace of $\ns$, and let $n\leqslant \dim(V)$. For any $n$-dimensional subspace $V_n\subseteq V$ we have
\begin{equation*}
\|P_{V_n}\|_{L_2}\geqslant\sqrt{\sum_{j=1}^{\infty} \lambda_j - \sum_{j=1}^{n} \lambda_{j,V}},
\end{equation*}
and $E_{n,V}$ is the unique optimal subspace. In particular
\begin{equation*}
\|P_{V}\|_{L_2}=\sqrt{\sum_{j=1}^{\infty} \lambda_j - \sum_{j=1}^{\dim{V}} \lambda_{j,V}}.
\end{equation*}

\end{cor}

This corollary has two consequences. At one hand, if we want to have a small 
Power Function, we need to choose a subspace $V$ which provides a good 
approximation of the true eigenvalues. On the other hand, when dealing with 
certain point based subspaces, we can control the decay of the Power Function 
depending on the number of points we are using, and this bound will provide a 
bound also on the convergence of the discrete eigenvalues to the true one. 
The last fact will be discussed in more detail 
in Section \ref{sec:asympEig}.

We remark that there is a relation between $E_n$ and $E_{n,V}$: as 
mentioned before, the optimal subspace $E_{n,V}$ is not the projection 
of $E_n$ into $V$, but is near to be its optimal approximation from $V$. 
To see this, observe that the operator $T_V$ is the projection of $T$ into $V$. 
In fact, given $K_V(x,y) = \sum_{j} \lambda_{j,V}\varphi_{j,V}(x) \varphi_{j,V}(y)$, we have for any $f\in L_2(\Omega)$
\begin{align*}\label{Tmcharact}
T_V f (x) &= \int_{\Omega}K_V(x,y) f(y) dy = \sum_{j}
\sqrt{\lambda_{j,V}}\varphi_{j,V}(x) 
(\sqrt{\lambda_{j,V}}\varphi_{j,V},f)_{L_2} \\
&= \sum_{j} \sqrt{\lambda_{j,V}}\varphi_{j,V}(x) (\sqrt{\lambda_{j,V}}\varphi_{j,V},T f )_{L_2} = \Pi_{\calh,V}T f(x).
\end{align*} 
This means that the couples $(\lambda_{j,V},\varphi_{j,V})$ are precisely the 
Bubnov - Galerkin approximations (see e.g. \cite[Sec. 18.4]{KVZRS72}) of the 
solutions $(\lambda_{j},\varphi_{j})$ of the eigenvalue problem for the
restricted 
operator $T : \ns\to \ns$ (which is still a compact, positive and self-adjoint
operator). 
We can then use the well known estimates on the convergence of the Bubnov -
Galerkin method 
to express the distance between $E_n$ and $E_{n,V}$.
 
The following Proposition collects 
convergence results which follow  
from \cite[Th. 18.5, Th. 18.6]{KVZRS72}.
\begin{prop}\label{prop:conv}
Let $V_1\subset V_2\subset \dots\subset V_n\subset\dots$ be a sequence of closed 
subspaces which become dense in $\ns$. For $1\leqslant j\leqslant \dim{V_n}$ we have
\begin{enumerate}[(i)]
\item $\lambda_{j,V_n}\leqslant \lambda_{j,V_{n+1}}\leqslant \lambda_j$,
\item Let $r\in\N$ be the multiplicity of $\lambda_j$ and \\
$F_{j,n} = \{f\in V_n : T_{V_n} f = \lambda_{i,V_n} f \text{ and }
  \lim_{n\to\infty}\lambda_{i,V_n} = \lambda_{j}  \}$. 
For $n$ sufficiently large, 
 $\dim F_{j,n} = r$ and  there exists 
$c_{j,n} > 1/\lambda_j,\ c_{j,n}\rightarrow_n 1/\lambda_j$ s.t.
\begin{equation}\label{BG-conv}
	\|\varphi_j - \Pi_{\calh,F_{j,n}} \varphi_j \|\leqslant c_{j,n} \lambda_j \| \varphi_j - \Pi_{\calh,V_n} \varphi_j\|.  
\end{equation}
\end{enumerate}
\end{prop}

Equation \eqref{BG-conv} proves in particular that $E_{n,V}$ is an 
asymptotically optimal approximation of $E_n$. Indeed, under the 
assumptions of the last Proposition, we have
\begin{equation}
\| \varphi_j - \Pi_{\calh,V_n} \varphi_j\|\leq	\|\varphi_j - \Pi_{\calh,F_{j,n}} \varphi_j \|\leqslant c  \| \varphi_j - \Pi_{\calh,V_n} \varphi_j\|,  
\end{equation}
with $c\to 1$ as $m\to\infty$.

\begin{rem}
To conclude this section we point out that, in addition to the point based sets,
there is another remarkable way to produce closed subspaces of the native
space. Namely, if $\Omega_1\subset \Omega$ is any 
Lipschitz subdomain of $\Omega$,
$\calh(\Omega_1)$ is a closed subset of $\ns$. This implies that the
eigenvalues are decreasing with respect to the inclusion of the base domain (as
can by proven also by the min/max characterization of the eigenvalues).
\end{rem}

\section{Asymptotic decay of the eigenvalues}\label{sec:asympEig}
We established a relation between the approximation error and 
the eigenvalues of $T$. This allows to use the well known bounds 
on the approximation error to give corresponding bounds on the decay of the eigenvalues. 
These results were already presented in \cite{SchWen2002}, but we include them
here for completeness and add some extensions.

We consider a set of asymptotically uniformly distributed points $X_n$, such that the fill distance $h_n$ behaves like
\begin{equation*}
h_n := \sup_{x\in\Omega}\min_{x_j\in X_n}\|x-x_j\|\leq c n^{-1/d},
\end{equation*}
where $c$ is independent of $n$.
 
If the kernel is translational invariant and Fourier transformable 
on $\R^d$ and $\Omega$ is bounded and with a smooth enough boundary, 
there are standard error estimates for the error between $f\in\ns$ and its interpolant $s_{f,X_n}$ on the points $X_n$ (see \cite{Sch1995}). 
 
For kernels $k(x-y) := K(x,y)$ with finite smoothness, we have that 
$\hat k(\omega) \sim (1+\|w\|)^{-\beta-d}$ for $\|w\|\to\infty$, and   
\begin{equation}\label{notLocalBound}
\|f - s_{f,X_n} \|_{\infty} \leq c h_{n}^{\beta/2}\|f\|,\; \fa f\in \ns,
\end{equation}
while for infinitely smooth kernels we have
\begin{equation*}
\|f - s_{f,X_n} \|_{\infty} \leq c \exp(-c /h_{n})\|f\|,\; \fa f\in \ns.
\end{equation*}
Both bounds are in fact bounds on the 
$L_{\infty}(\Omega)$ norm of the Power Function. If instead one considers directly the 
$L_2(\Omega)$ error, for kernels with finite smoothness the estimate can be improved as follows:
\begin{equation*}
\|f - s_{f,X_n} \|_{L_2} \leq c h_{n}^{(\beta + d)/2}\|f\|,\; \fa f\in \ns.
\end{equation*}

This immediately leads to the following theorem.

\begin{thm}\label{TheEigDecay}
Under the above assumptions on $K$ and $\Omega$, the eigenvalues decay at least like 
\begin{equation*}
\sqrt{\lambda_{n+1}} < c_1 n^{-(\beta + d) / 2 d}
\end{equation*}
for a kernel with smoothness $\beta$, and at least like 
\begin{equation*}
\sqrt{\lambda_{n+1}} < c_2 \exp(-c_3 n^{1/d}),
\end{equation*}
for kernels with unlimited smoothness.  
The constants $c_1, c_2, c_3$ are independent of $n$, but dependent on $K$, $\Omega$, and the space dimension.
\end{thm} 

It is important to notice that the asymptotics of the eigenvalues is 
known for the kernels of limited smoothness, and on $\R^d$. 
If the kernel is of order $\beta$, its native space on $\R^d$ is norm 
equivalent to the Sobolev space $H^{(\beta+d)/2}$. In these spaces the
$n$-width, 
and hence the eigenvalues, decay like $\Theta(n^{-\frac{\beta+d}{2d}})$ (see
\cite{Jer1970}). 
This means that in Sobolev spaces one can recover (asymptotically) 
the best order of approximation using kernel spaces.

This can be done also with point based spaces. The following 
statement follows from Corollary \ref{L2minVm}, applying the same ideas as
before. 
Observe that, in this case, we need to consider the bound \eqref{notLocalBound}.
\begin{cor}\label{TheEigApprox}
Under the above assumptions on $K$ and $\Omega$, we have 
\begin{equation*}
0 \leq \lambda_j - \lambda_{j,V(X_n)} < c_1 n^{-\beta / d}, \;\; 1\leq j\leq n,
\end{equation*}
for a kernel with smoothness $\beta$, and at least like 
\begin{equation*}
0 \leq \lambda_j - \lambda_{j,V(X_n)} < c_2 \exp(-c_3 n^{1/d}), \;\; 1\leq j\leq n,
\end{equation*}
for kernels with unlimited smoothness.  
The constants $c_1, c_2, c_3$ are independent of $n$, but dependent on $K$, $\Omega$, and the space dimension.
\end{cor}

This Corollary and the previous Theorem proves that, using point based sets with
properly chosen points, 
one can achieve at the same time a fast decay of the true eigenvalues and a fast convergence of the discrete ones.

Both results in this section raise some questions about the converse implication. 
From Theorem \ref{TheEigDecay} we know that the smoothness of the kernel 
guarantees a fast decay of the eigenvalues. But we can also start from a 
given expansion to construct a kernel. Is it possible to conclude smoothness of 
the kernel from fast decay of the eigenvalues? 
 
Corollary \ref{TheEigApprox} tells us that uniformly distributed 
point based sets provide a good approximation of the eigenvalues. We will see 
in Section \ref{sec:alg} and \ref{sec:num} that one can numerically construct 
point based sets whose eigenvalues are close to the true ones. 
Is it possible to prove that these sets are 
necessarily asymptotically uniformly distributed?

\section{Algorithms}\label{sec:alg}

If we consider a closed subspace $V_N=V(X_N)$ spanned by translates 
of the kernel on a set $X_N$ of $N$ points in $\Omega$, it is 
possible to explicitly construct $E_{n,V_N}$, $n\leq N$. 

The number $N$ of points should be large enough to simulate $L_2(\Omega)$ inner products by discrete formulas. 
The first method we present aims at a direct construction 
of the eigenspaces and gives some insight on the structure 
of the basis, but it is numerically not convenient since it 
involves the computation of two Gramians.
For stability and computability reasons, we shall then focus on 
lower-dimensional subspaces of $V(X_N)$. There are various ways 
to construct these, and we present two. We will see in Section 
\ref{sec:num} that these approximation are very effective. 

From now on we will use the subscript $N$ in place 
of $V(X_N)$ to simplify the notation, keeping in mind 
that there is an underlying set of points $X_N\subset \Omega$ 
and the corresponding subspace $V(X_N)$.

\subsection{Direct construction}
We use at first the fact that the eigenbasis is the 
unique set of $N$ functions in $V(X_N)$ 
which is orthonormal in $L_2(\Omega)$ and orthogonal in $\ns$, 
where uniqueness is understood in the sense of uniqueness 
of the eigendecomposition of the integral operator 
\eqref{discreteTOperator}.

Given any couple of inner products $(\cdot,\cdot)_a$ and $(\cdot,\cdot)_b$ 
on $V(X_N)$, it is always possible to build a basis $\{v_j\}_{j=1}^N$ 
of $V(X_N)$ which is $b$-ortho{\em normal}
and $a$-ortho{\em gonal}
with norms $\{\sigma_j\}_{j=1}^N$. Let 
\begin{align*}
A =& [( K(\cdot,x_i),K(\cdot,x_j))_a]_{i,j=1}^N,\\
B =& [(K(\cdot,x_i),K(\cdot,x_j))_b]_{i,j=1}^N
\end{align*}
be the Gramians with respect to the two inner products of the standard 
basis $K(\cdot,x_1)$, $\dots$, $K(\cdot, x_N)$ of $K$ translates. 
Following the notation of \cite{PaSc2011}, to construct the basis we 
need to construct an invertible matrix $C_V$ of change of basis to express 
this new basis with respect to the standard basis. To have the right
orthogonality, 
we need $C_V^T A C_V = \Sigma$ and $C_V^T B C_V = I$, where $\Sigma$ is the 
diagonal matrix having on the diagonal the $a$-norms of the new basis. This 
means to simultaneously diagonalize the two Gramians, and since they are 
symmetric and positive definite this is always possible, e.g. in the following way:
\begin{itemize}
\item $B = L L^T$ be a Cholesky decomposition,
\item define $C = L^{-1} A L^{-T}$ (which is symmetric and positive definite),
\item let $C = U \Gamma U^T$ be a SVD decomposition,
\item define $C_V = L^{-T} U$.
\end{itemize}
Observe that, for practical use, it is more convenient to swap the 
role of $A$ and $B$. In this way we construct the basis 
$\{\sqrt{\lambda_{j,N}}\varphi_{j,N}\}_{j=1}^N$, which is $\ns$-orthonormal 
hence more suitable for approximation purposes, and, moreover, 
we obtain directly the eigenvalues of order $N$ as 
$\Sigma = \text{diag}(\lambda_{j,N},j=1,\dots,N)$. 
 
In both ways we are just computing a generalized diagonalization 
of the pencil $(A,B)$, up to a proper scaling of the diagonals.
In our case $A$ is the usual kernel matrix, while $B_{ij} = (K(\cdot,x_i),K(\cdot,x_j))_{L_2}$.
Thus, provided we know $B$, we can explicitly construct the basis. 
This can be done, for a general kernel, using a large set of 
points to approximate the $L_2(\Omega)$ inner product. 
 
Numerical experiments (see Section \ref{sec:num}) suggest 
that this construction of the eigenspaces is highly unstable 
also in simple cases, since it requires to solve an high 
dimensional matrix eigenvalue problem. Another way to face the problem 
consists of greedy procedures as considered next. 

\subsection{Greedy approximation} \label{sec:greedy}
Instead of directly constructing the subspace $E_{n,N}$ via $N\times N$ matrix eigenvalue problems as described before, we can first select $n$ points in $X_N$ such that working on these points is more stable than working with the full original matrix, and then solve the problem in $V(X_n)$.
The selection of the set $X_n$ is performed with a greedy construction of the Newton basis (see \cite{MuSch2009, DMScWe2005}).

First we show how to construct the eigenspaces via the Newton basis. Assume that $v_1,\ldots,v_n$ is the Newton basis for $V(X_n)$. Then
$$
T_{n}f(x)=\sum_{i=1}^n v_i(x)\sum_{j=1}^n(v_i,v_j)_{L_2}(f,v_j)
$$
and if $\lambda_{j,n}$ is an eigenvalue  with eigenfunction $\varphi_{j,n}$ then
$T_{n}\varphi_{j,n}= \lambda_{j,n}\varphi_{j,n} $ implies 
$$
\lambda_{j,n}(\varphi_{j,n},v_i)=\sum_{k=1}^n (v_i,v_k)_{L_2}(\varphi_{j,n},v_k).
$$
Thus the coefficients of the eigenbasis with respect to the Newton basis are the eigenvectors of the $L_2(\Omega)$ Gramian of the Newton basis. Experimentally, the Newton basis is nearly $L_2(\Omega)$ orthogonal.
Thus the above procedure should have a nice Gramian matrix, provided that the $L_2(\Omega)$ inner products that are near zero can be calculated without loss of accuracy.

To select the points we use two similar greedy strategies, based on maximization of $L_{\infty}(\Omega)$ and $L_2(\Omega)$ norms of the Power Function. The first point is chosen as
\begin{equation*}
x_1=\argmax_{x\in X_N} \left\|\frac{K(\cdot,x)}{\sqrt{K(x,x)}}\right\|_{L_{\infty}}\;\text{ or }\; x_1=\argmax_{x\in X_N} \left\|\frac{K(\cdot,x)}{\sqrt{K(x,x)}}\right\|_{L_2}.
\end{equation*}
Denoting by $X_i$ the already chosen points, the $(i+1)$-th point is selected as 
\begin{equation*}
x_{i+1} = \argmax_{x\in X_N\setminus X_i}  \|v_{i+1}\|_{L_{\infty}}^2\;\text{ or }\; x_{i+1} = \argmax_{x\in X_N\setminus X_i}  \|v_{i+1}\|_{L_2}^2.
\end{equation*}
The point sets selected by the two strategies are different, but we will see in the next section that they provide similar results in terms of approximation of the eigenspaces.

\section{Experiments}\label{sec:num}

\subsection{Sobolev kernels}

In this section we consider the Mat\'ern kernels of order $\beta = 0, 1, 2, 3$, 
whose native spaces on $\R^d$ are norm equivalent to the Sobolev spaces $H^{(\beta + d) / 2}$. 

In these spaces the asymptotic behavior of the Kolmogorov width, 
hence of the eigenvalues, is known as recalled in Section \ref{sec:asympEig}. 
We assume here that the same bounds  
hold in a bounded domain (the unit disk), 
and we want to compare it with the discrete eigevalues of point based sets, 
which can be computed with one of the methods of the previous section. 

To this aim, we start from a grid of equally spaced points in $[-1, 1]^2$ 
restricted to the unit disk, so that the number of points inside the disk 
is $m \approx 10^4$. We use this grid both for point selection and to 
approximate the $L_2(\Omega)$ inner products as weighted pointwise products, i.e.,
$$
(f,g)_{L_2}\approx \frac{\pi}{m}\sum_{j=1}^m f(x_j) g(x_j).
$$
We select $n = 200$ points by the greedy $L_{\infty}(\Omega)$ maximization of the Power Function in the unit disk. 
 
The eigenvalues are then computed with the method of Section \ref{sec:greedy}, i.e, as eigenvalues of the $L_2(\Omega)$ Gramian of the Newton basis.

The results are shown in Figure \ref{fig:MatOrder}. As expected, for any order 
under consideration there exist a positive constant $c$ such that the discrete 
eigenvalues decay with the same rate of the Sobolev best approximation.

Moreover, we expect that the discrete eigenvalues converge to the true 
ones with a rate that can also be controlled by $\beta$. Indeed, 
according to Corollary \ref{TheEigApprox}, we have
\begin{equation*}
0 \leq \lambda_j - \lambda_{j,V(X_n)} < c_1 n^{-\beta / d}, \;\; 1\leq j\leq n.
\end{equation*}
To verify this, we instead look at the decay of 
$$
\sum_{j=0}^{\infty} \lambda_j - \sum_{j=0}^{n} \lambda_{j,V(X_n)}.
$$
since we can exactly compute the first term. Indeed, since the kernels are radial, we have
$$
\sum_{j=1}^{\infty}\lambda_j = \int_{\Omega} K(x,x) dx = \pi K(0,0).
$$
Results are presented in Figure \ref{fig:MatDiff}. From the experiments 
it seems that we can  
improve the convergence speed somewhat,  
and in fact obtain a rate of 
order $(\beta + d/2)/d$ instead of $\beta /d$.
This may be another instance of the ``gap of $d/2$'' already 
observed in \cite{schaback-wendland:2002-1}. Sometimes, observed convergence
rates are by $d/2$ better than proven ones.

\begin{figure}[hbt]
\centering
\begin{tabular}{cc}
\includegraphics[width=0.5 \textwidth]{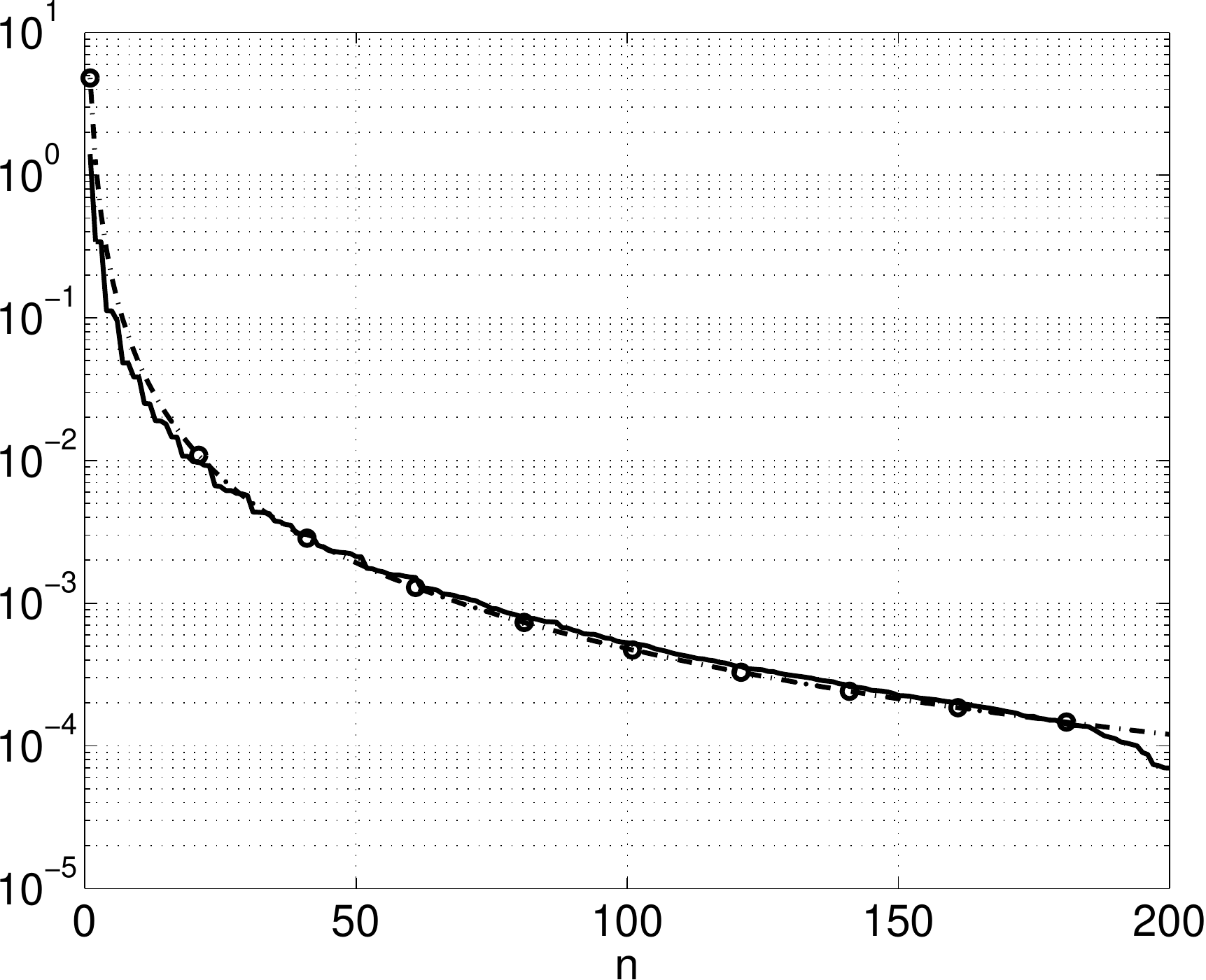}&
\includegraphics[width=0.5 \textwidth]{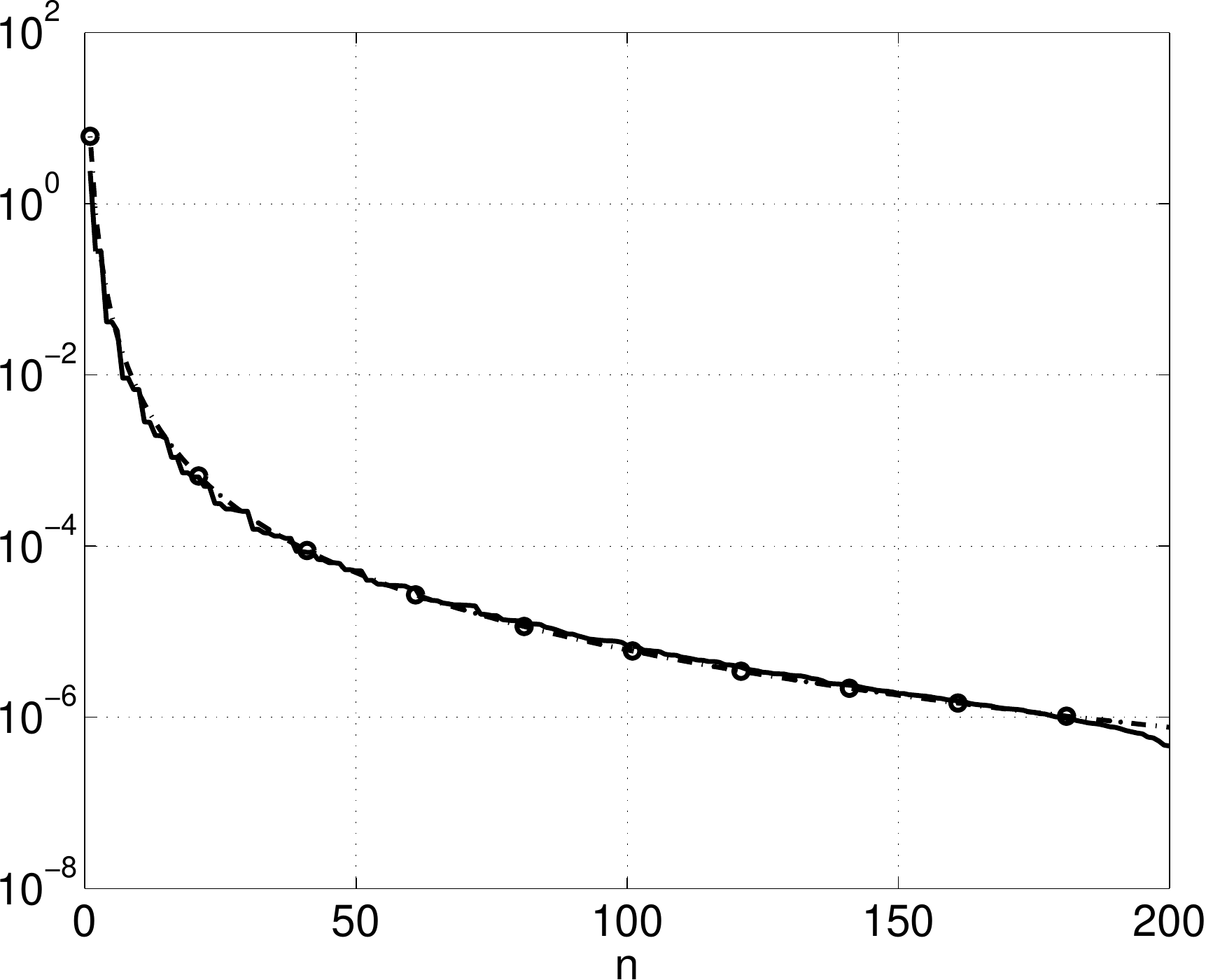}\\

\includegraphics[width=0.5 \textwidth]{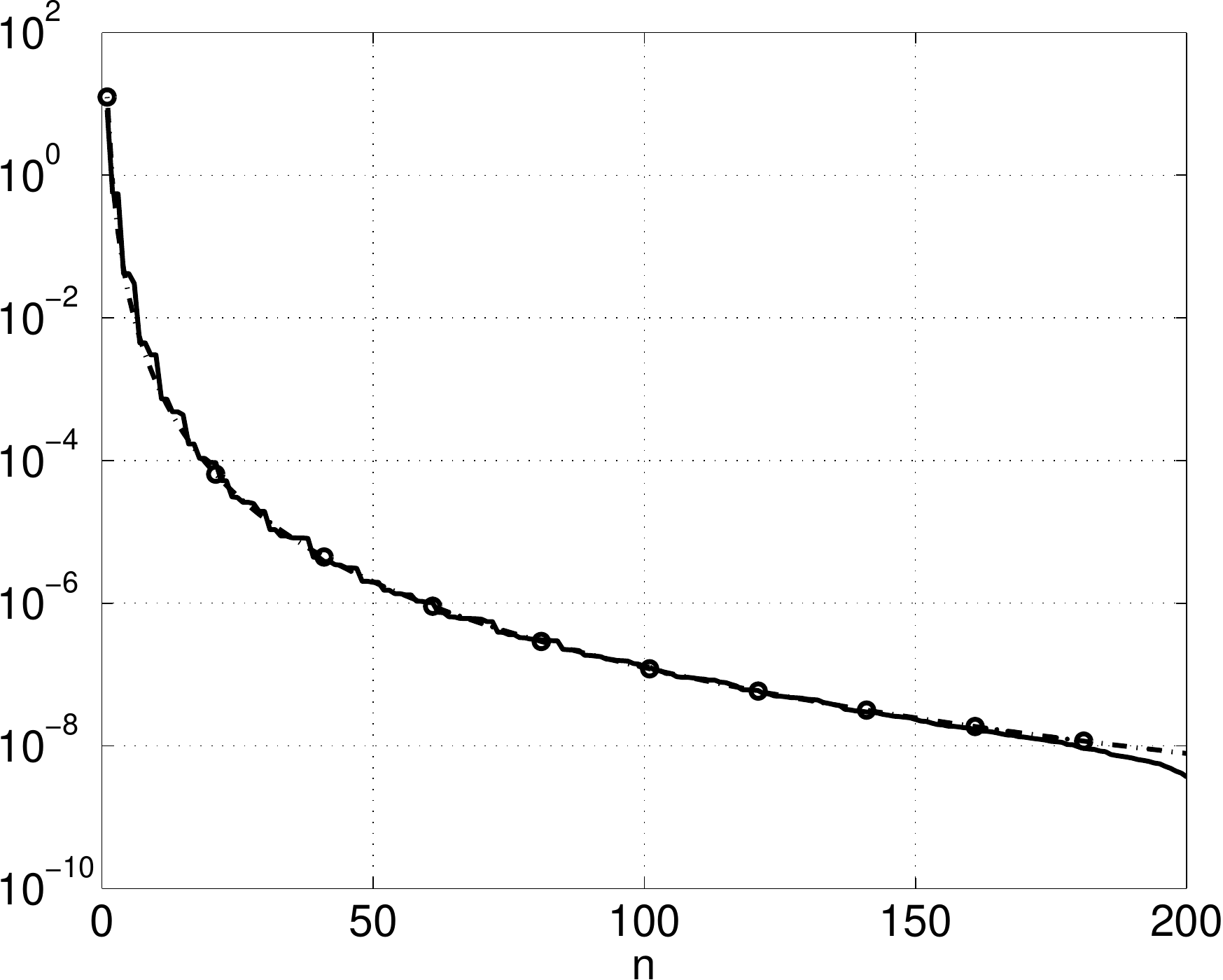}&
\includegraphics[width=0.5 \textwidth]{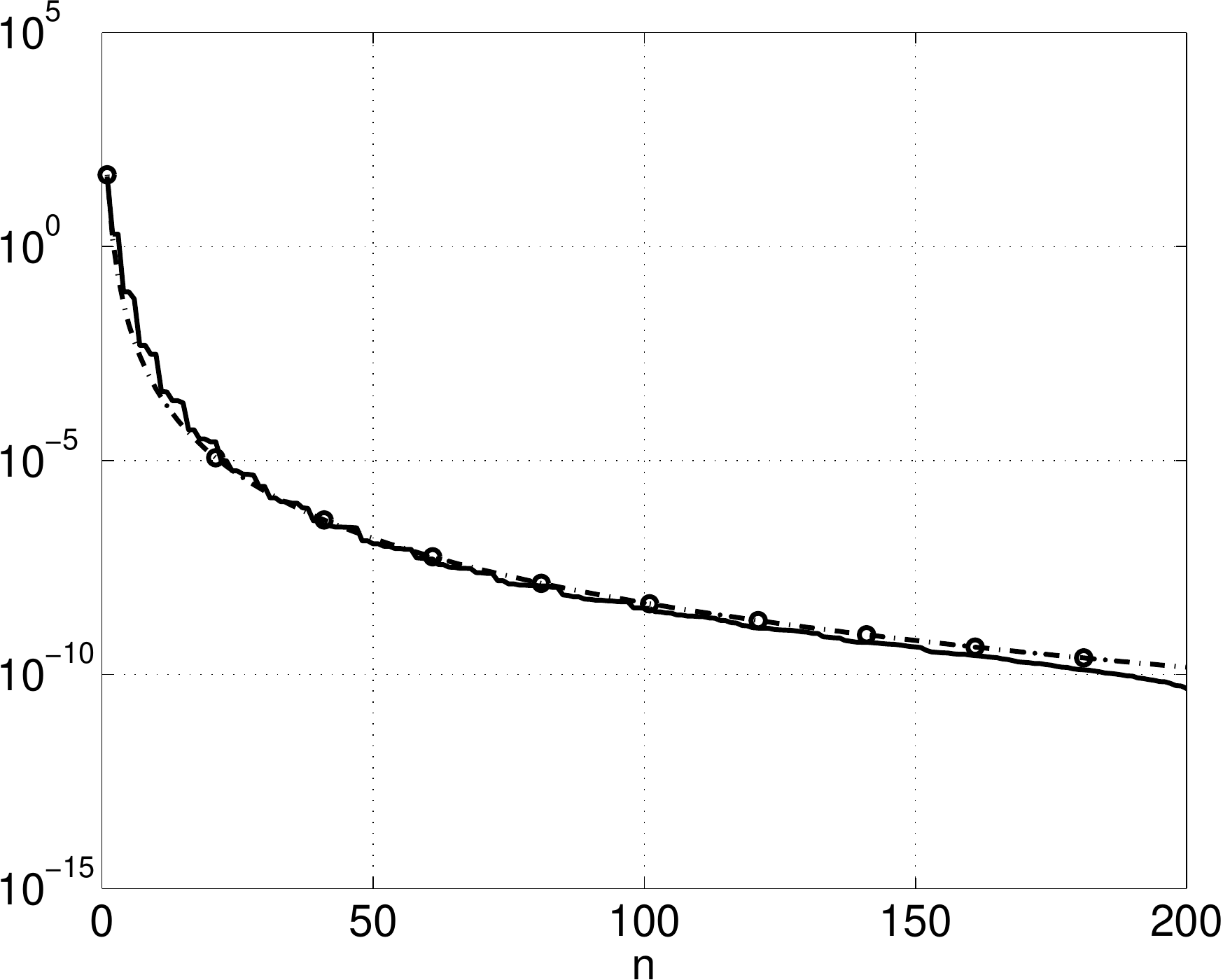}
\end{tabular}
\caption{Decay of the discrete eigenvalues of the Matern kernels (solid line) compared with the theoretical decay rate $n^{-(\beta+d)/d}$ in the corresponding Sobolev spaces (circles). The theoretical bounds are scaled with a positive coefficient. From top left to bottom right: $\beta = 0,1,2,3$.}\label{fig:MatOrder}
\end{figure}

\begin{figure}[hbt]
\centering
\begin{tabular}{cc}
\includegraphics[width=0.5 \textwidth]{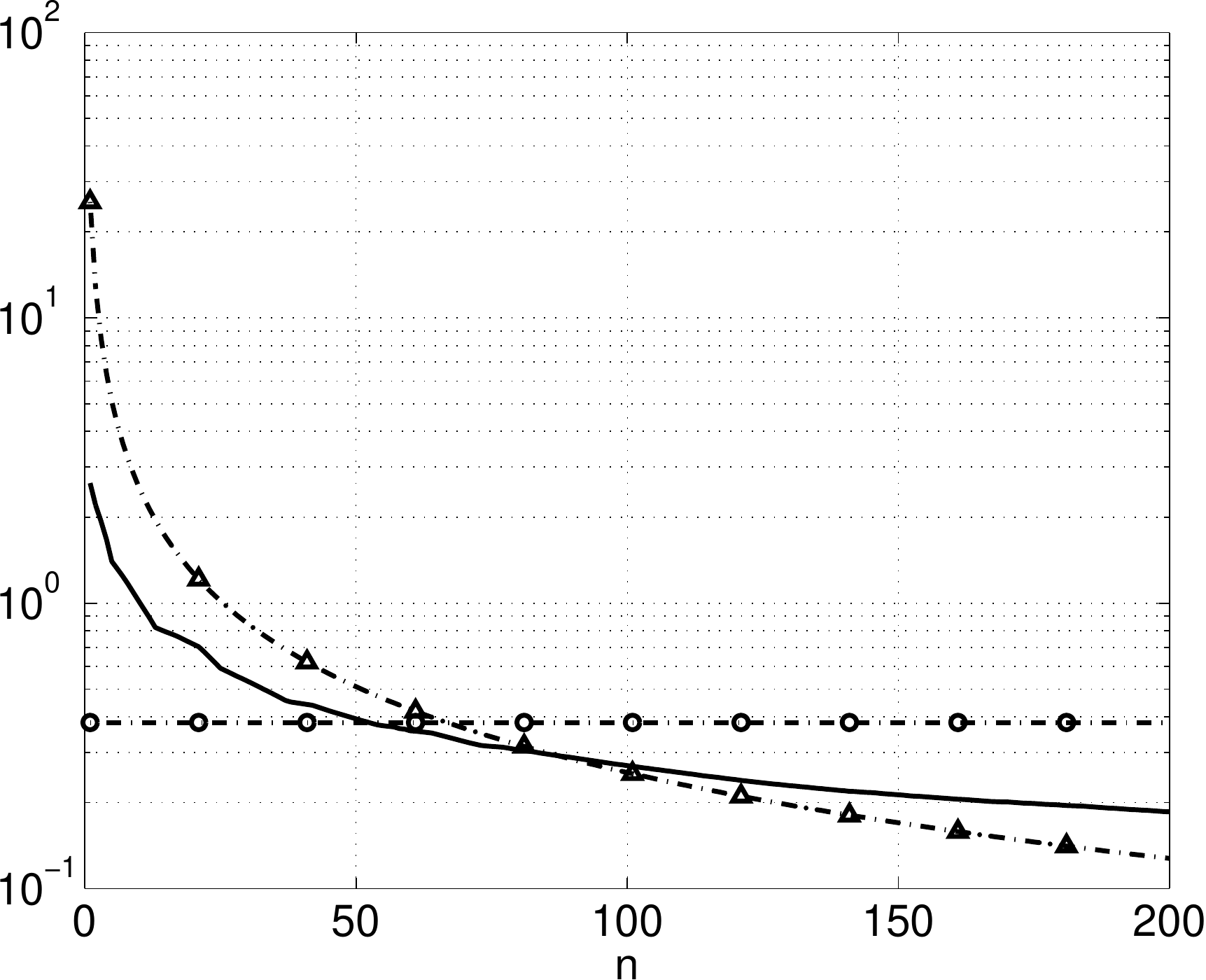}&
\includegraphics[width=0.5 \textwidth]{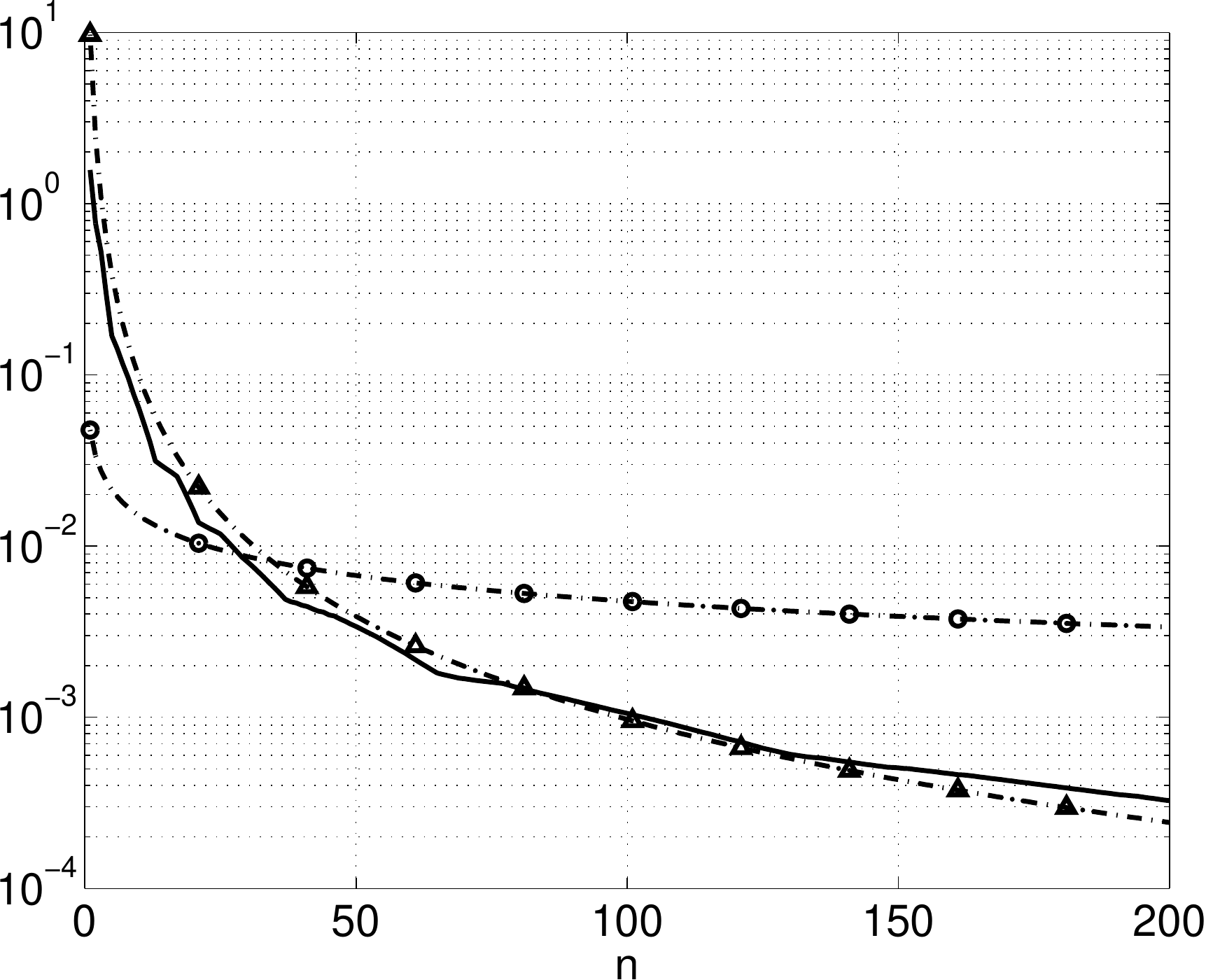}\\

\includegraphics[width=0.5 \textwidth]{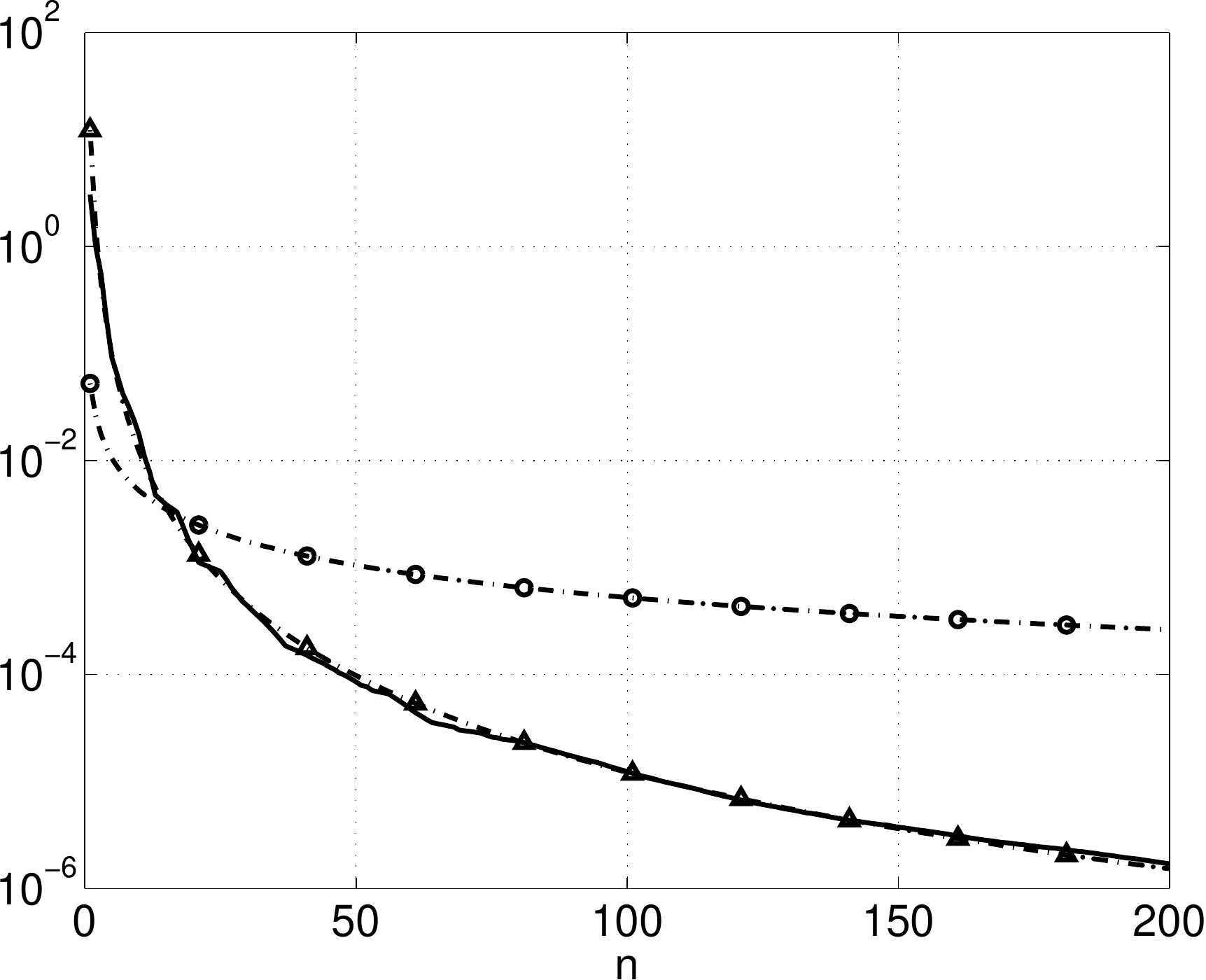}&
\includegraphics[width=0.5 \textwidth]{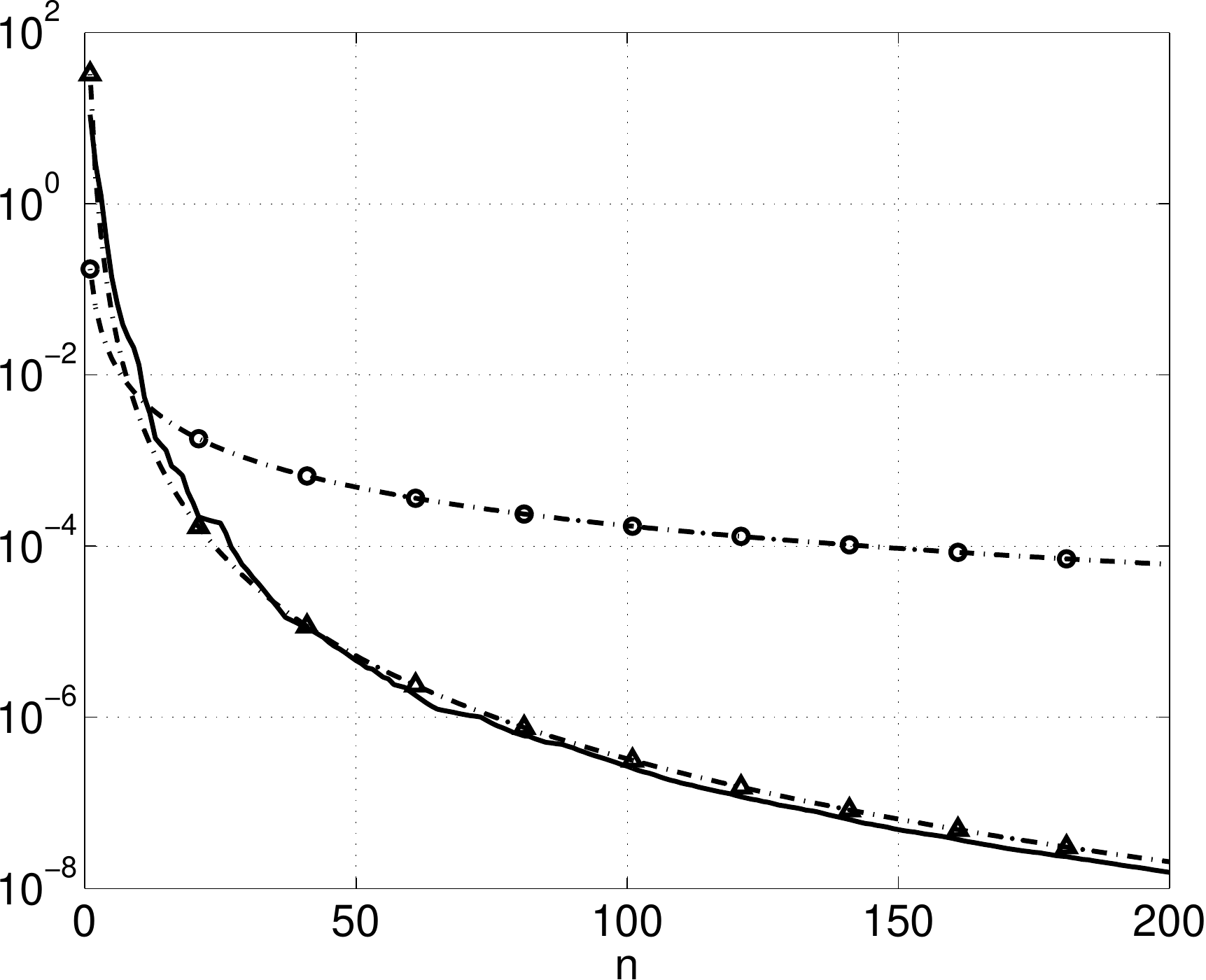}
\end{tabular}
\caption{Difference between the sum of the real eigenalues and the discrete ones (solid line), for the Matern kernels, compared with the theoretical decay rate $n^{-\beta/d}$ in the corresponding Sobolev spaces (circles) and with $n^{-(\beta + d/2)/d}$ (triangles). The theoretical bounds are scaled with a positive coefficient. From top left to bottom right: $\beta = 0,1,2,3$.}\label{fig:MatDiff}
\end{figure}

\subsection{Brownian bridge kernels}\label{sec:BBker}
In this section we  
experiment with the \textit{iterated Brownian bridge kernels} 
(see \cite{CaFaMcC2014}). This family of kernels is useful for 
our purposes because the exact eigenbasis is explicitly known 
and the smoothness of the kernel can be controlled using a parameter.
 
The kernels are defined, for $\beta\in\N\setminus\{0\}$, 
$\varepsilon\geqslant 0$ and $x,y\in[0,1]$, as
\begin{equation}\label{def:BB}
K_{\beta,\varepsilon}(x,y) = \sum_{j=1}^{\infty} \lambda_j(\varepsilon,\beta) \varphi_j(x)\varphi_j(y),
\end{equation}
where
\begin{equation*}
\lambda_j(\varepsilon,\beta) = \left(j^2\pi^2+\varepsilon^2\right)^{-\beta},\quad \varphi_j(x) = \sin(j\pi x).
\end{equation*}
The kernel has $2\beta - 2$ smooth derivatives. 
 
For $\beta = 1$ and $\varepsilon = 0$, the kernel has the form $K_{1,0}(x,y) = \min(x,y) - xy$, but a general closed form is not known. In the following tests we will compute it using a truncation of the series \eqref{def:BB} at a sufficiently large index. For simplicity, we will consider for now only $\varepsilon = 0,1$.
 
Thanks to the knowledge of the explicit expansion, we can also compute the $L_2(\Omega)$ - Gramian of $K_{\beta,\varepsilon}$ by squaring the eigenvalues in \eqref{def:BB}, i.e., 
\begin{equation}\label{def:BB2}
K_{\beta,\varepsilon}^{(2)}(x,y) = (K_{\beta,\varepsilon}(x,\cdot)K_{\beta,\varepsilon}(\cdot,y))_{L_2} = \sum_{j=1}^{\infty} \lambda_j(\varepsilon,2 \beta) \varphi_j(x)\varphi_j(y),
\end{equation}

First, we want to compare the optimal decay of the power function with the one
obtained by starting from a set of points $X_N$, both in the direct and in the
greedy way. We take $N = 500$ randomly distributed points in $(0,1)$ and we
construct an approximation of the eigenspace $E_{n,V_m}$ for $n = 50$. 
 
To speed up the algorithm, for the greedy selection we approximate the 
$L_2(\Omega)$ inner product with the discrete one on $X_N$. This step, of 
course, introduces an error in the selection of the points. Nevertheless, 
in order to evaluate properly the performance of the algorithm, after the 
construction the Newton basis we compute the $L_2(\Omega)$ Gramian exactly, i.e., using \eqref{def:BB2}. 
 
The results for $\beta = 1,\dots,4$ and $\varepsilon = 0,1$ are shown in 
Figure \ref{fig:power}. Lines not present in the plots mean that the 
corresponding power function is negative, because of numerical instability.
 
First, notice that the direct method is sufficiently stable only for 
$\beta = 1$, but in this case it is able to nearly reproduce the optimal 
rate of decay. The greedy algorithms, on the other hand, are feasible for 
all the choices of the parameters, and they have a convergence between 
$\|P_{E_{n,m}}\|_{L_2}$ and $\|P_{E_{2 n,m}}\|_{L_2}$. As the kernel becomes
smoother 
their performance is better, but they become unstable. Observe also that the 
two greedy selections of the points behave essentially in the same way, even 
if the first one is not designed to minimize any $L_2(\Omega)$ norm. 
 
Unlike for the smoothness parameter $\beta$, 
the dependence  
of the performance of the algorithm 
on $\varepsilon$ is not clear, 
except in the case 
$\beta = 4$, where $\epsilon = 1$ gives a better stability.

\begin{figure}[hbt]
\centering
\begin{tabular}{cc}
\includegraphics[width=0.5 \textwidth]{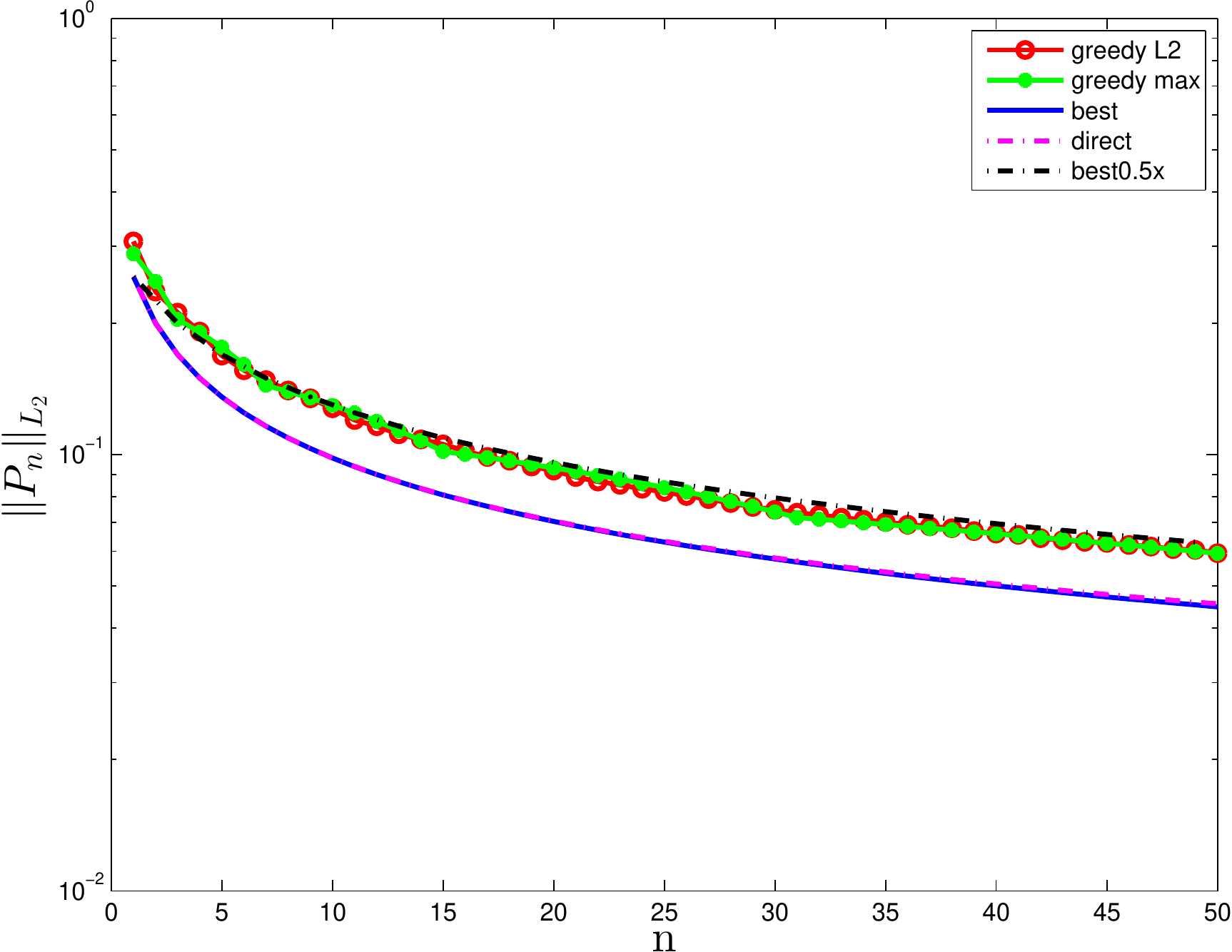}&
\includegraphics[width=0.5 \textwidth]{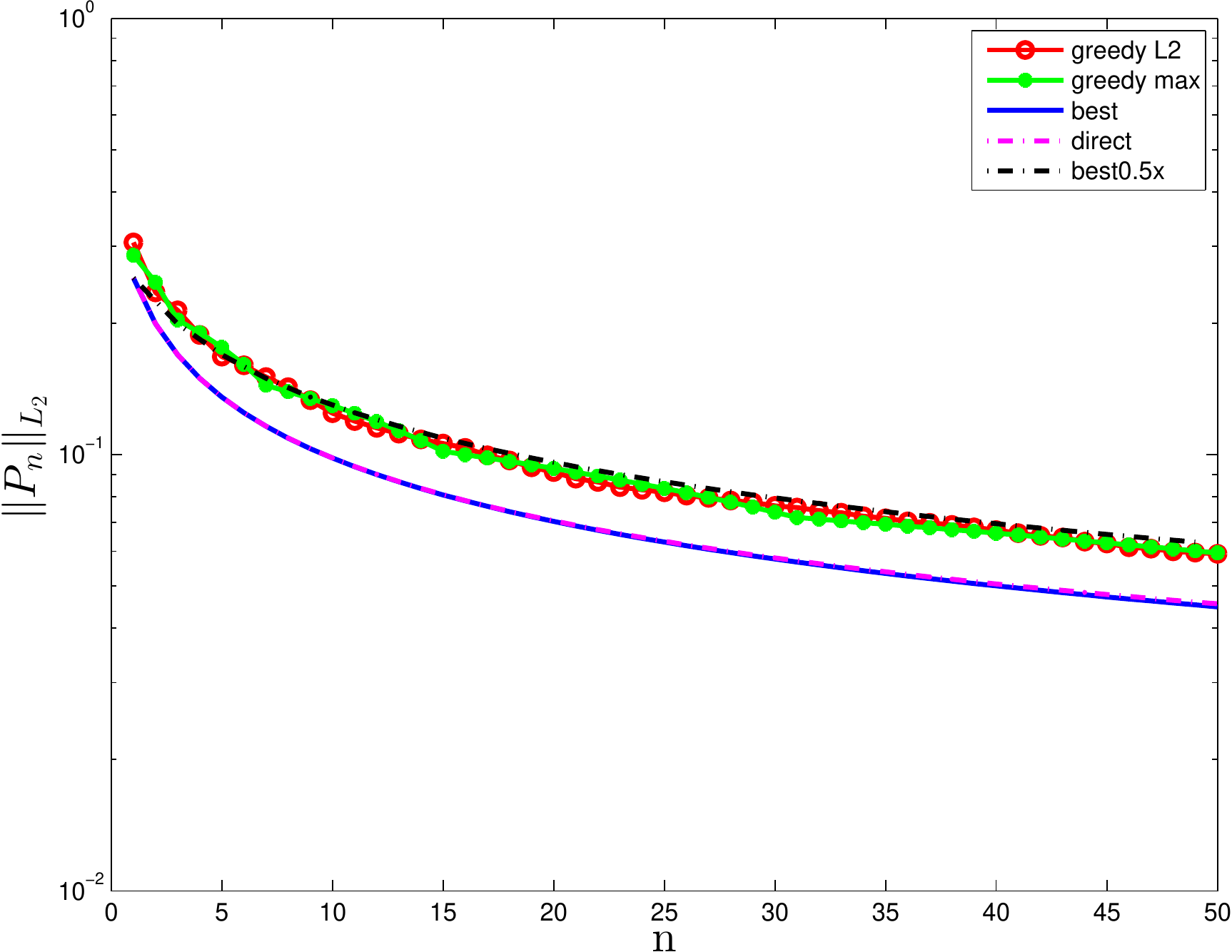}\\

\includegraphics[width=0.5 \textwidth]{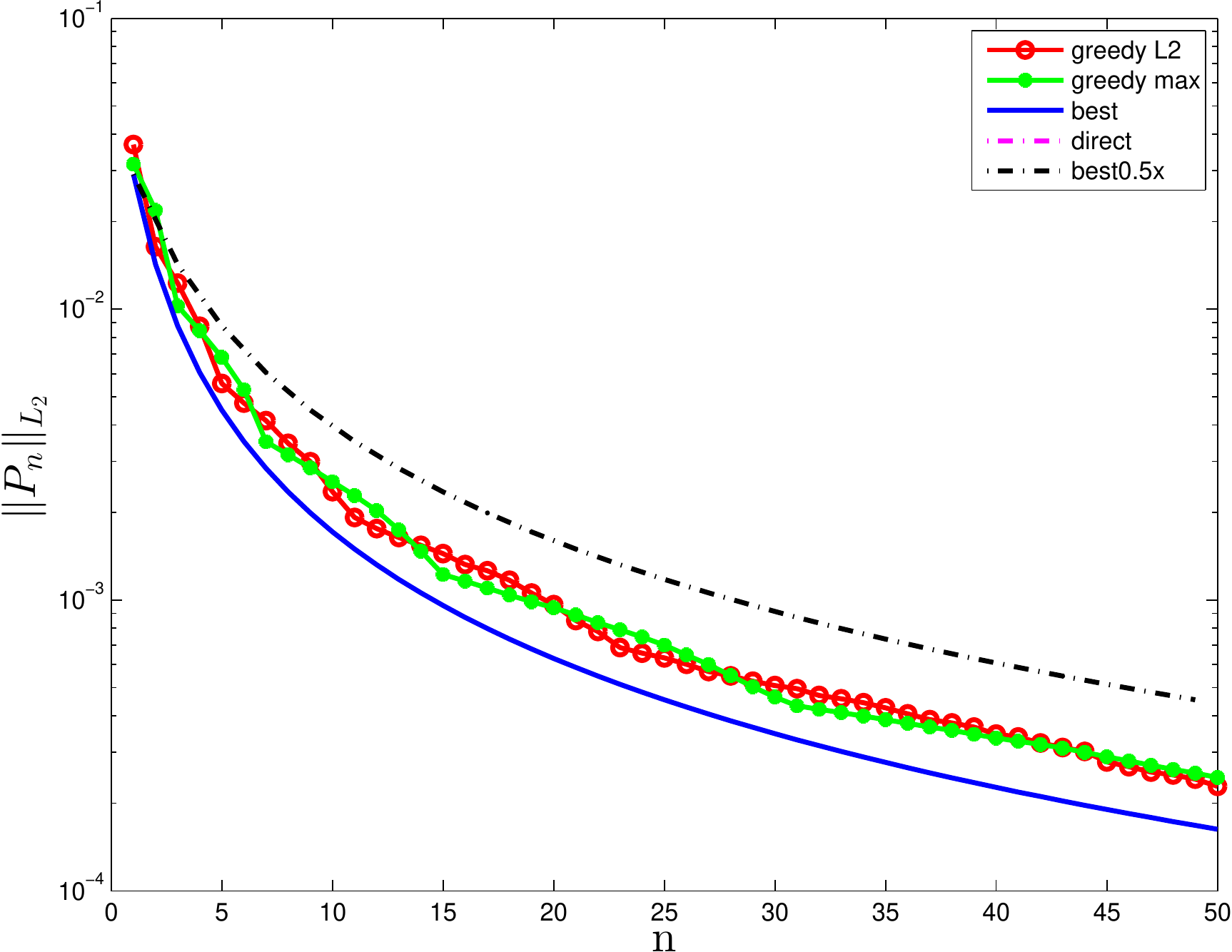}&
\includegraphics[width=0.5 \textwidth]{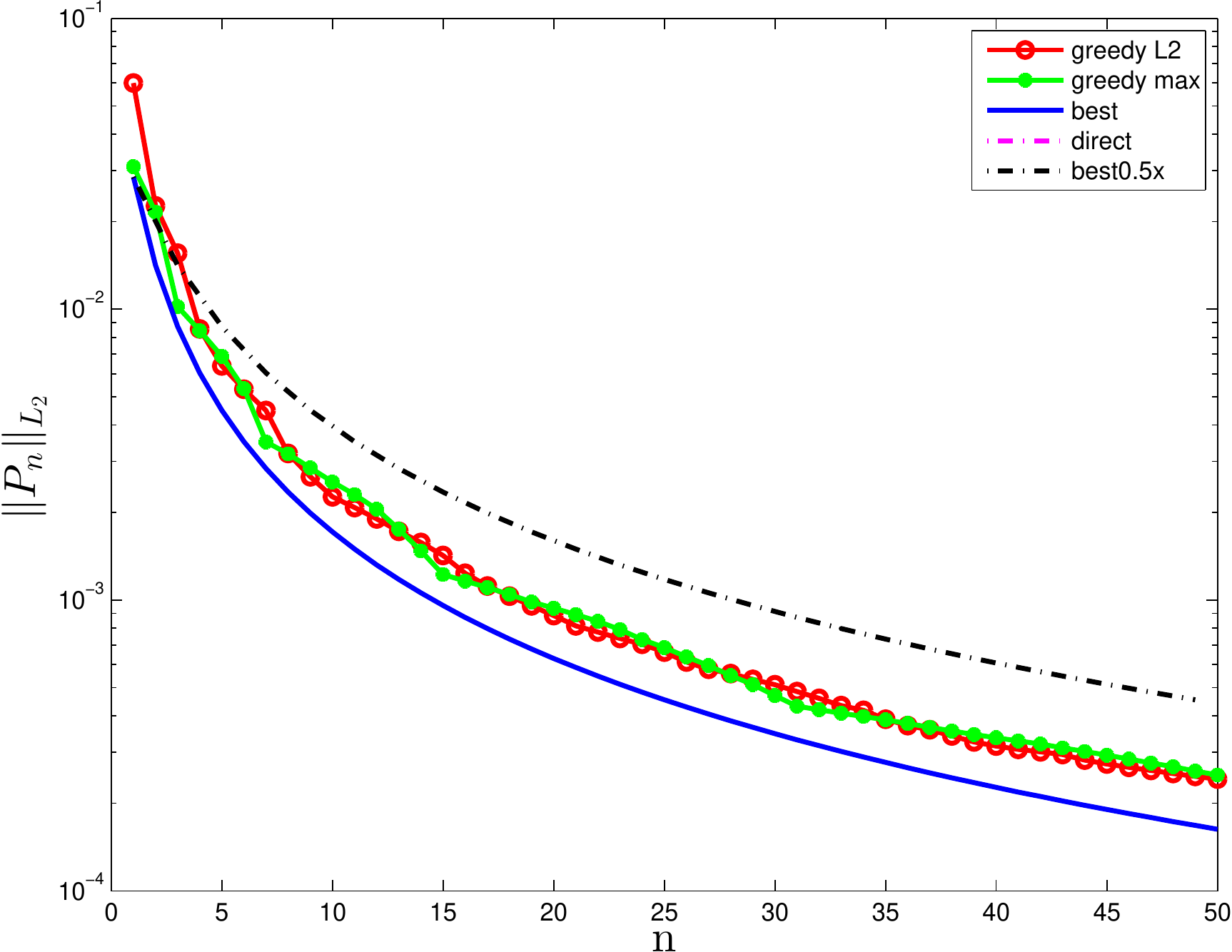}\\

\includegraphics[width=0.5 \textwidth]{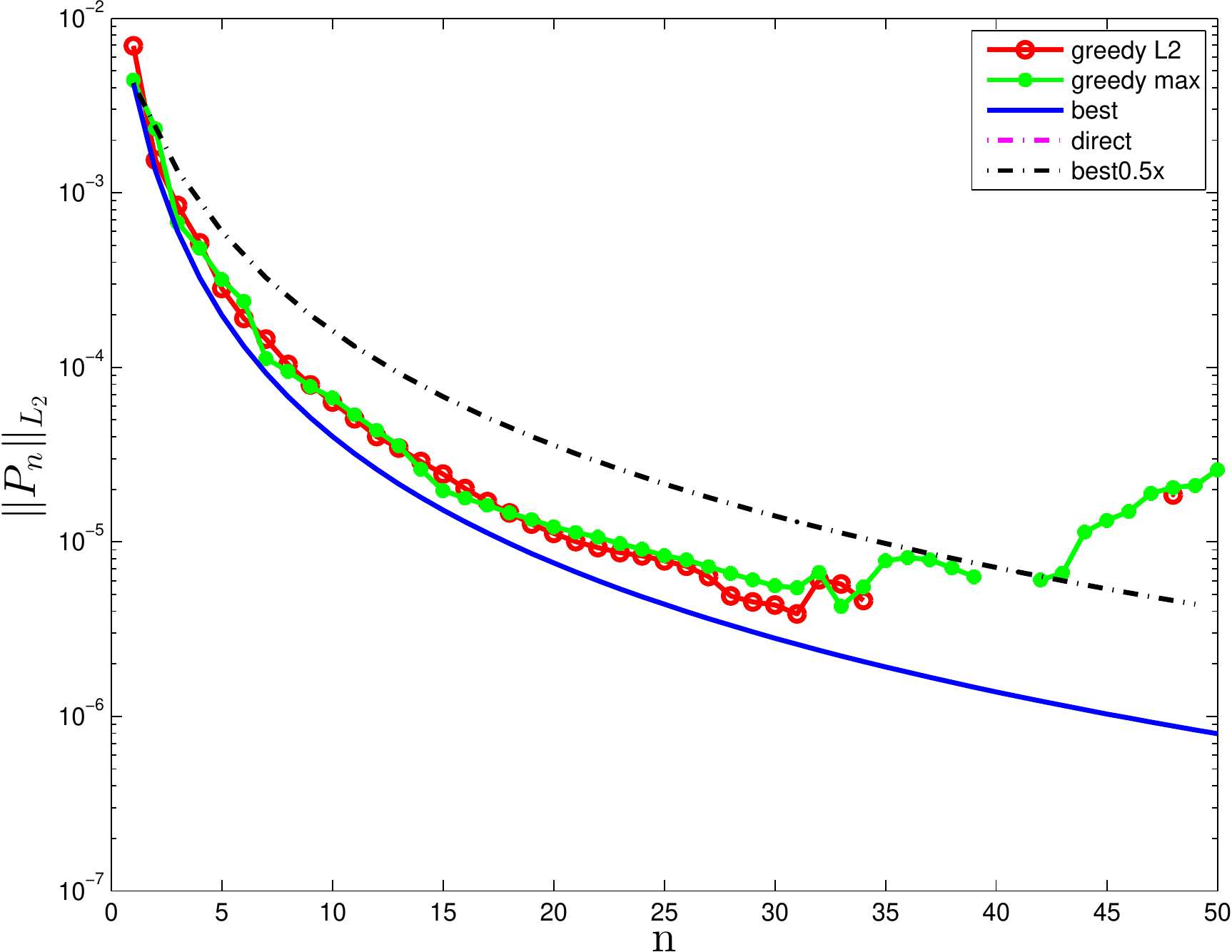}&
\includegraphics[width=0.5 \textwidth]{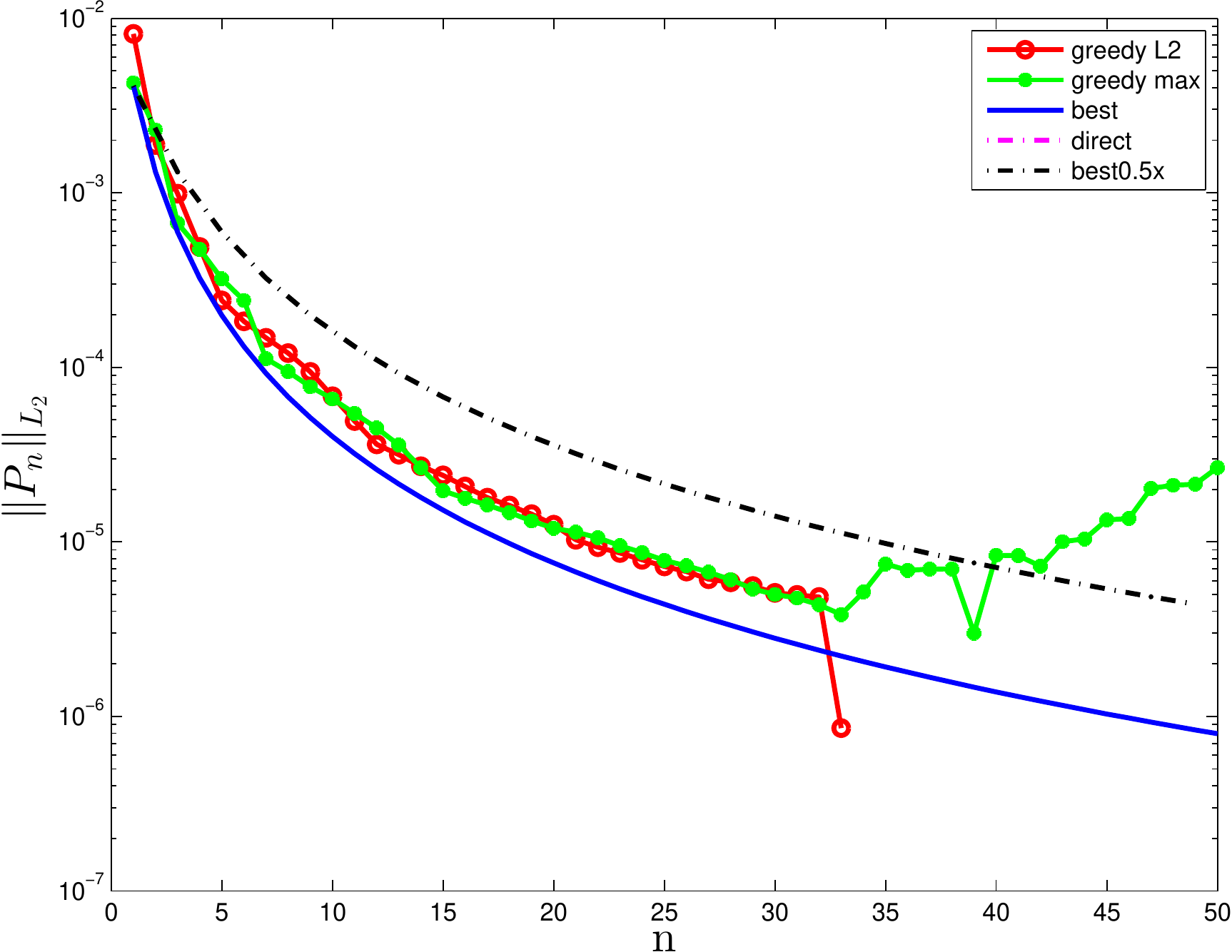}\\

\includegraphics[width=0.5 \textwidth]{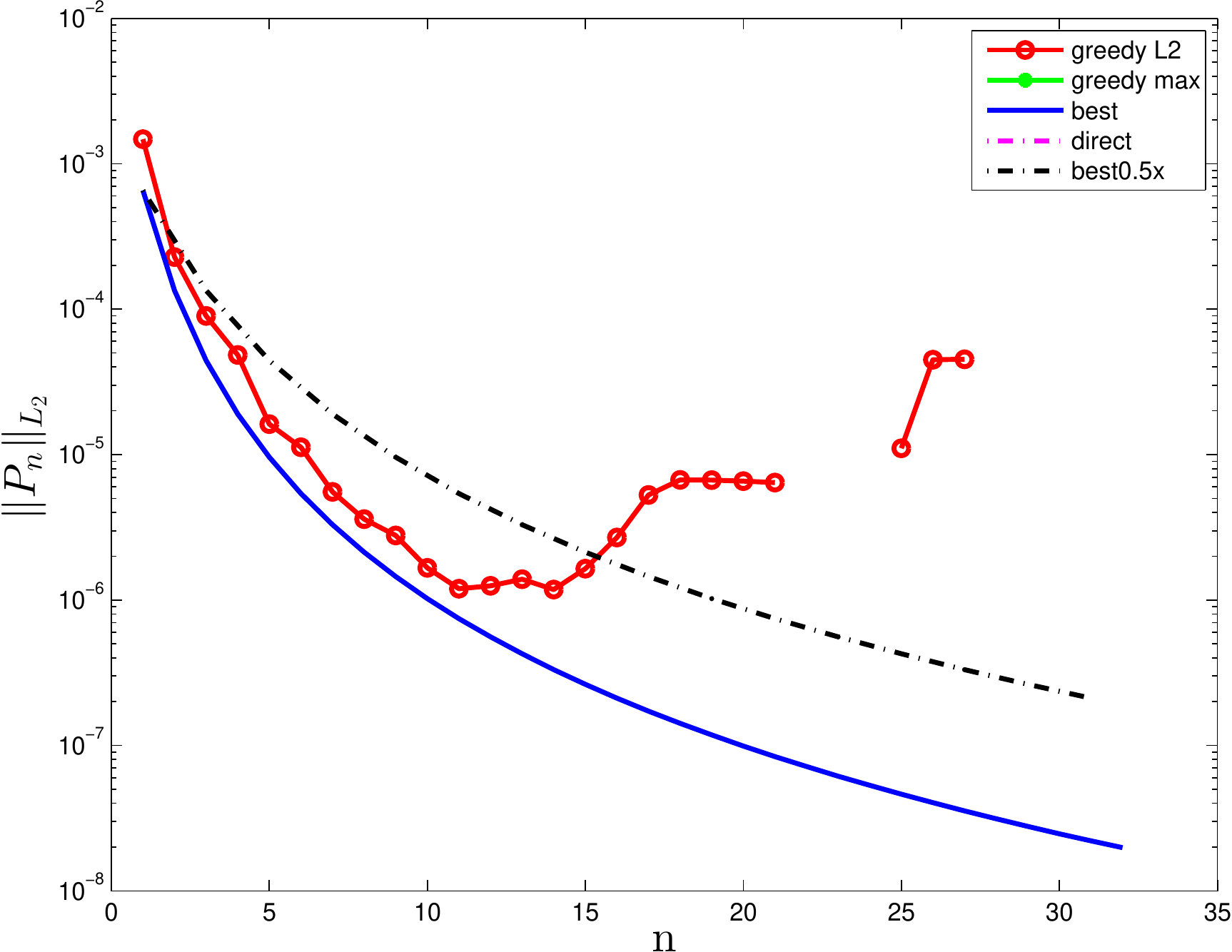}&
\includegraphics[width=0.5 \textwidth]{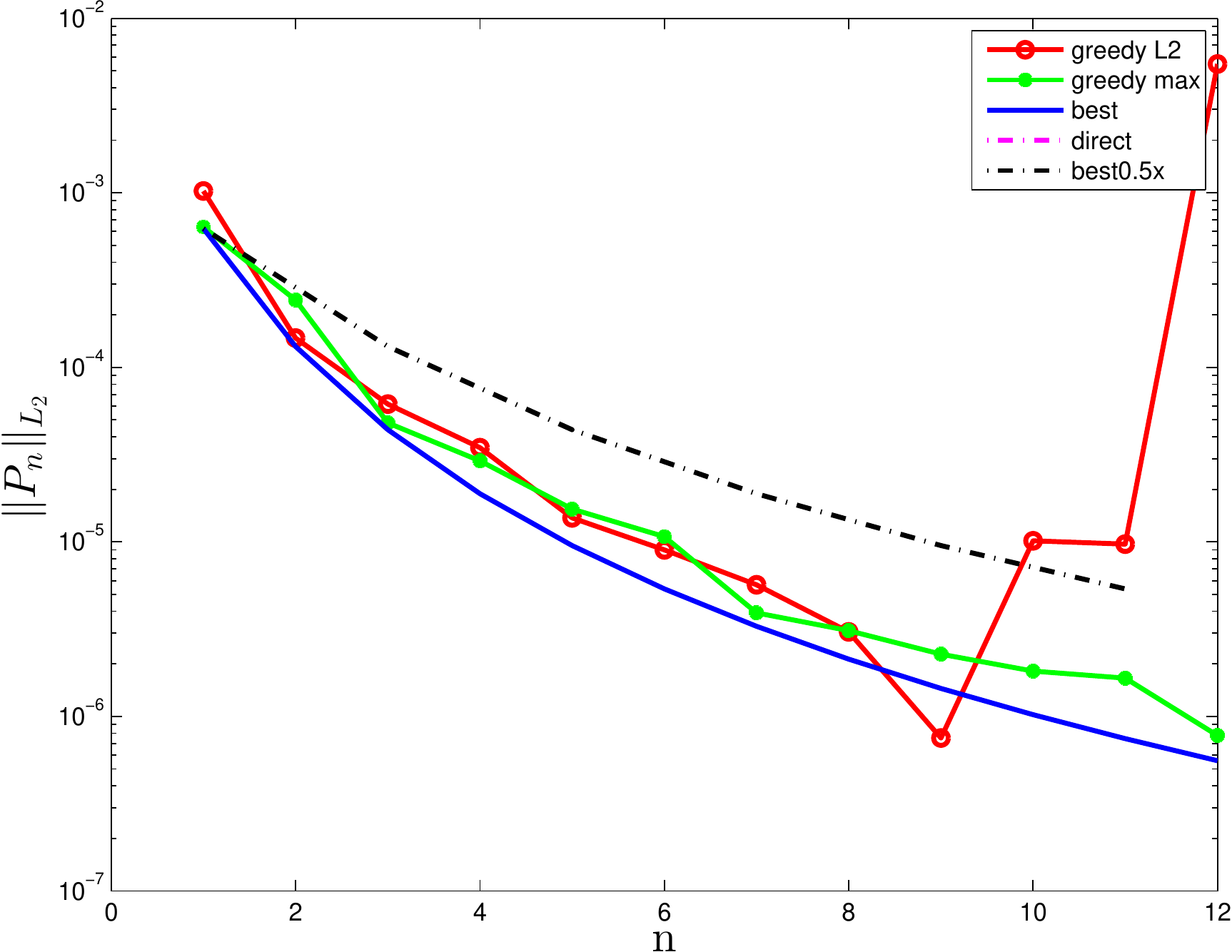}
\end{tabular}
\caption{Decay of the power functions described in Section \ref{sec:BBker} for different parameters (from left to right: $\varepsilon = 0, 1$; from top to bottom: $\beta = 1,2,3,4$). }\label{fig:power}
\end{figure}

Finally,  
we test the convergence of the approximate eigencouples to the exact ones.
 
Since the method becomes unstable, we use here a smaller set of 
$N = 100$ randomly distributed points in $(0,1)$ and we check 
the approximation for the first $n = 50$ eigenelements. Figure 
\ref{fig:approxEig} displays the results for $\beta = 1,2,3,4$ and 
$\varepsilon = 0$. Since the eigenbasis is defined up to a change of 
sign, we can expect to approximate $|\sqrt{\lambda_{j}}\varphi_j|$, 
$1\leqslant j\leqslant n$. 
 
As expected from Proposition \ref{prop:conv}, both for the eigenvalues 
and the eigenfunctions the convergence is faster for smoother kernels  and for
the smaller 
indices  
$j$. 
 
Also in this example the approximation becomes unstable for $\beta = 4$.

\begin{figure}[hbt]
\centering
\begin{tabular}{cc}
\includegraphics[width=0.5 \textwidth]{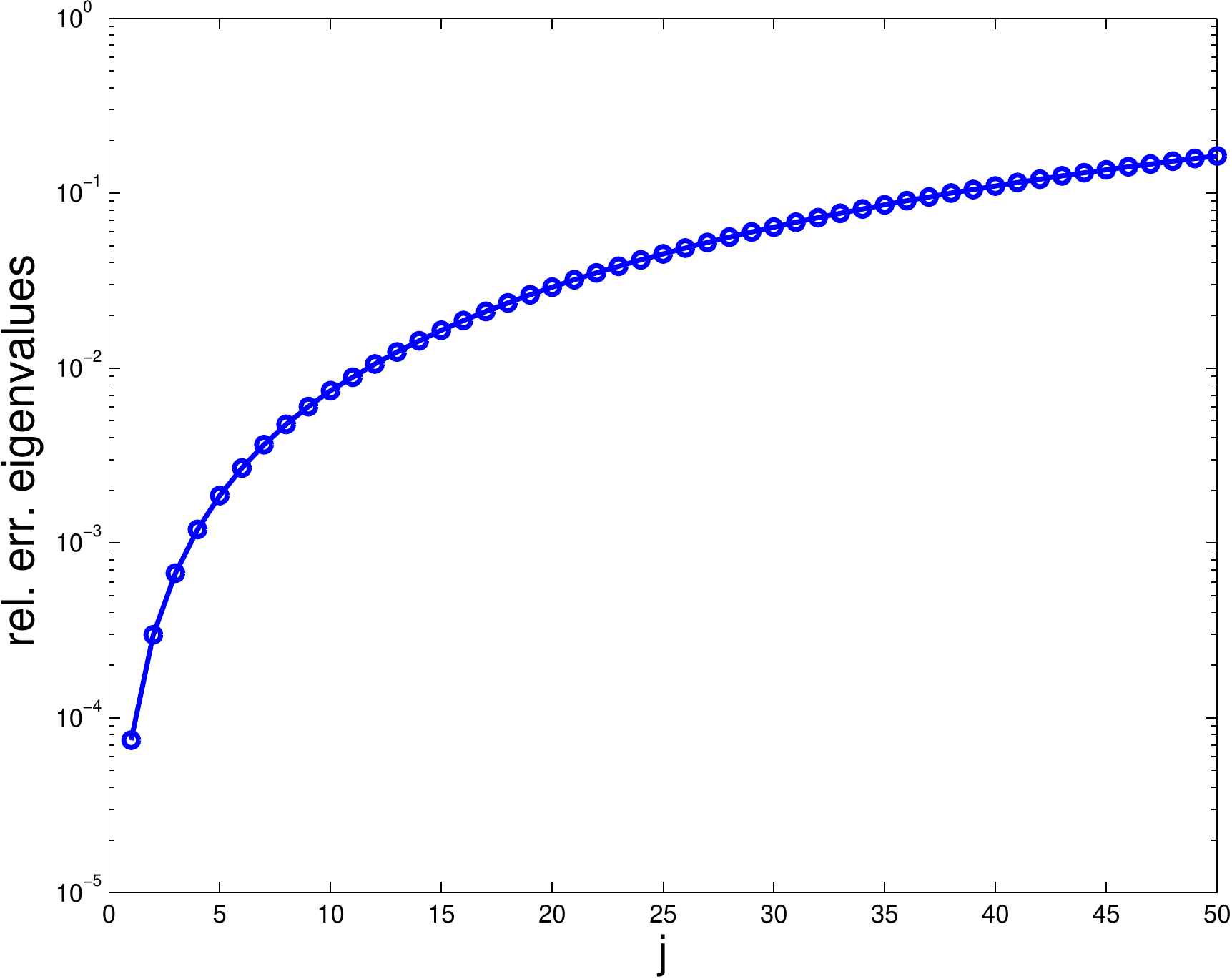}&
\includegraphics[width=0.5 \textwidth]{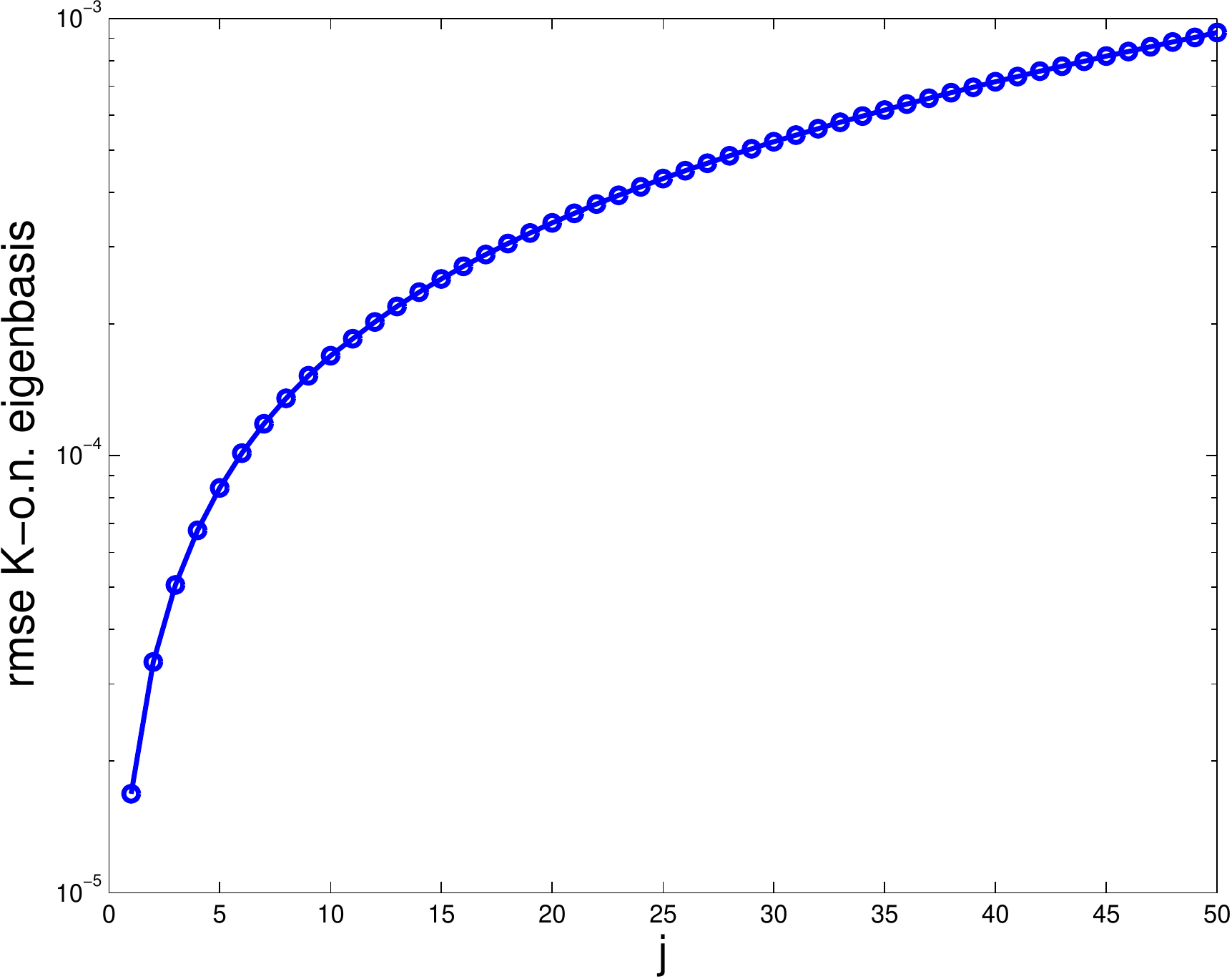}\\

\includegraphics[width=0.5 \textwidth]{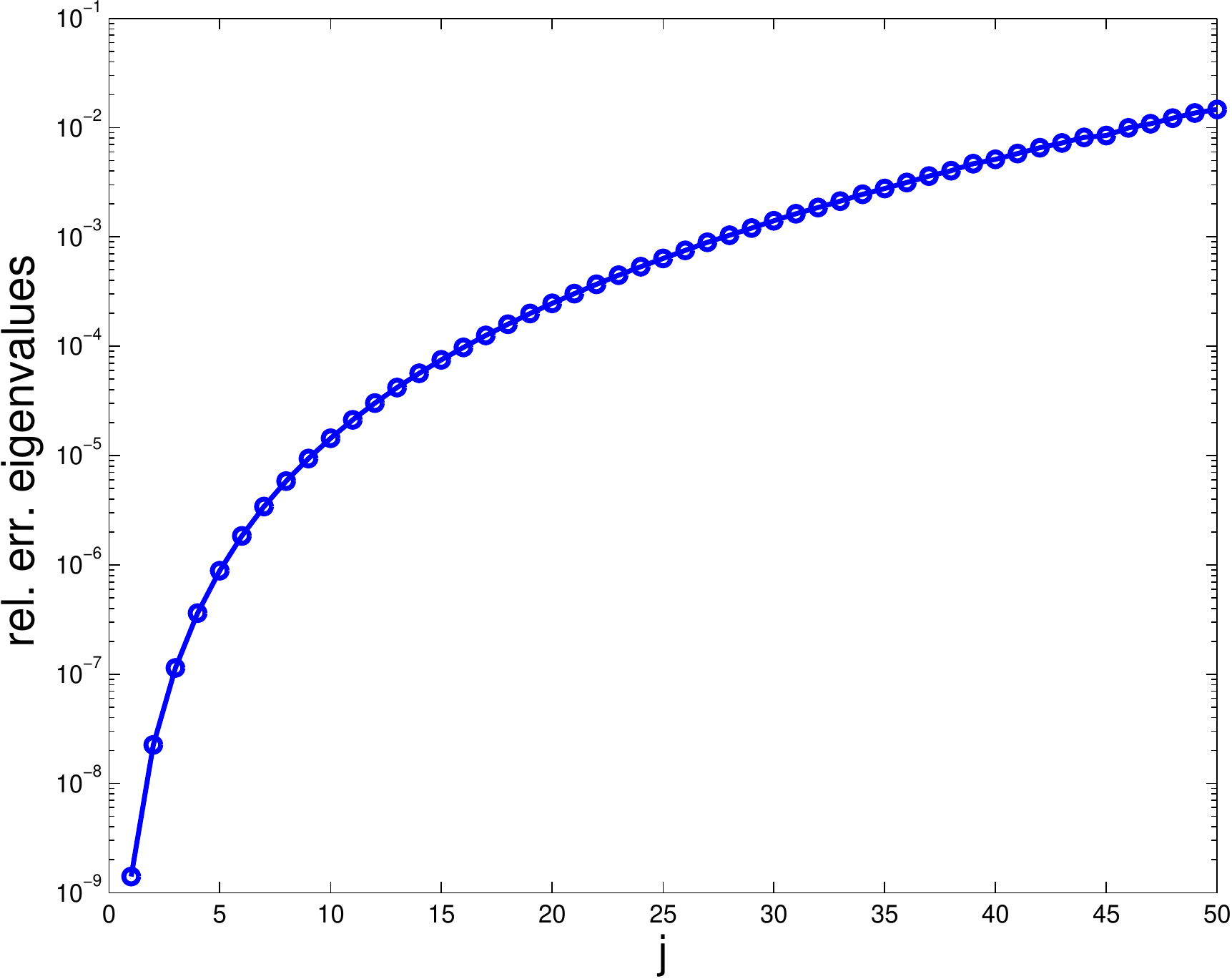}&
\includegraphics[width=0.5 \textwidth]{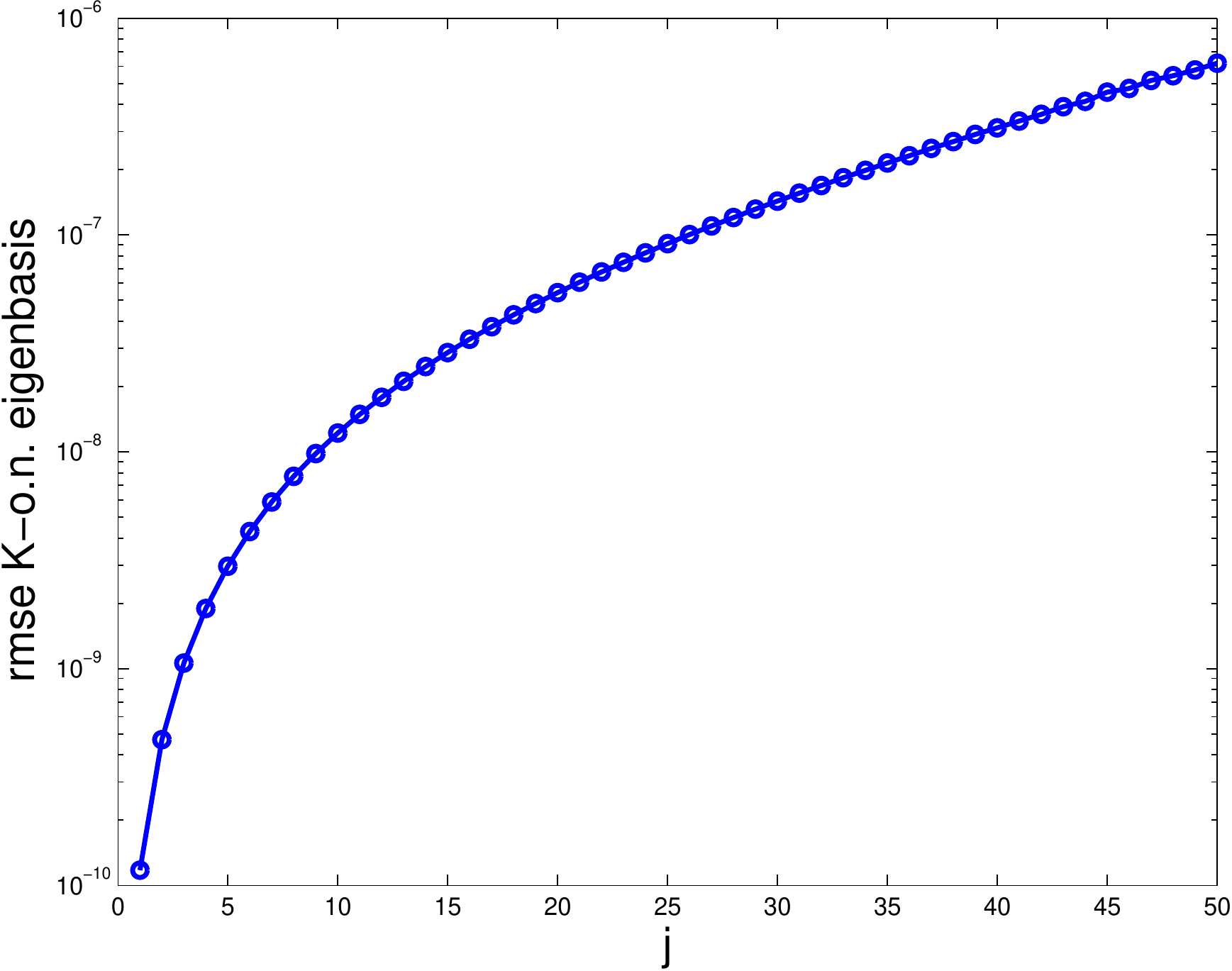}\\

\includegraphics[width=0.5 \textwidth]{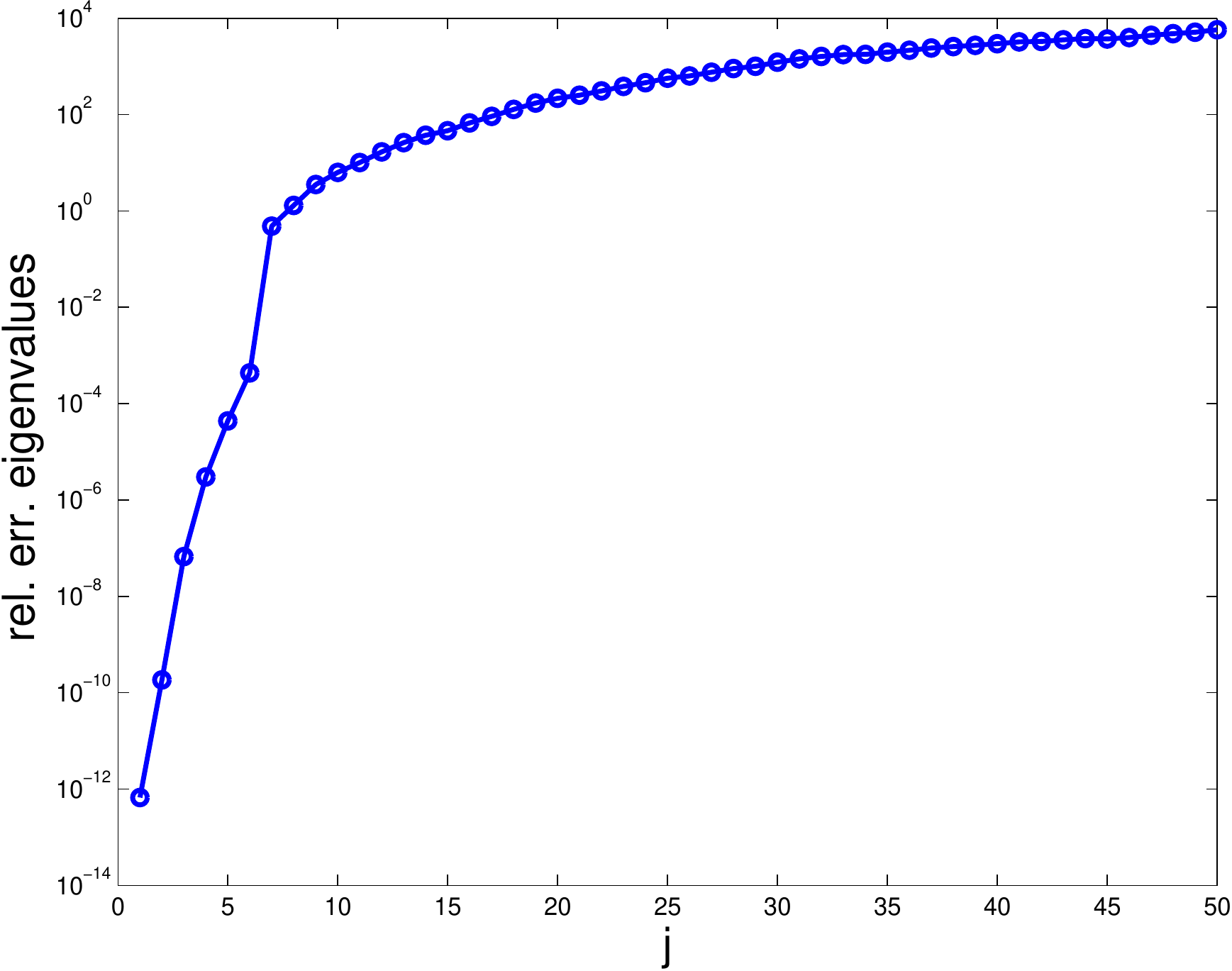}&
\includegraphics[width=0.5 \textwidth]{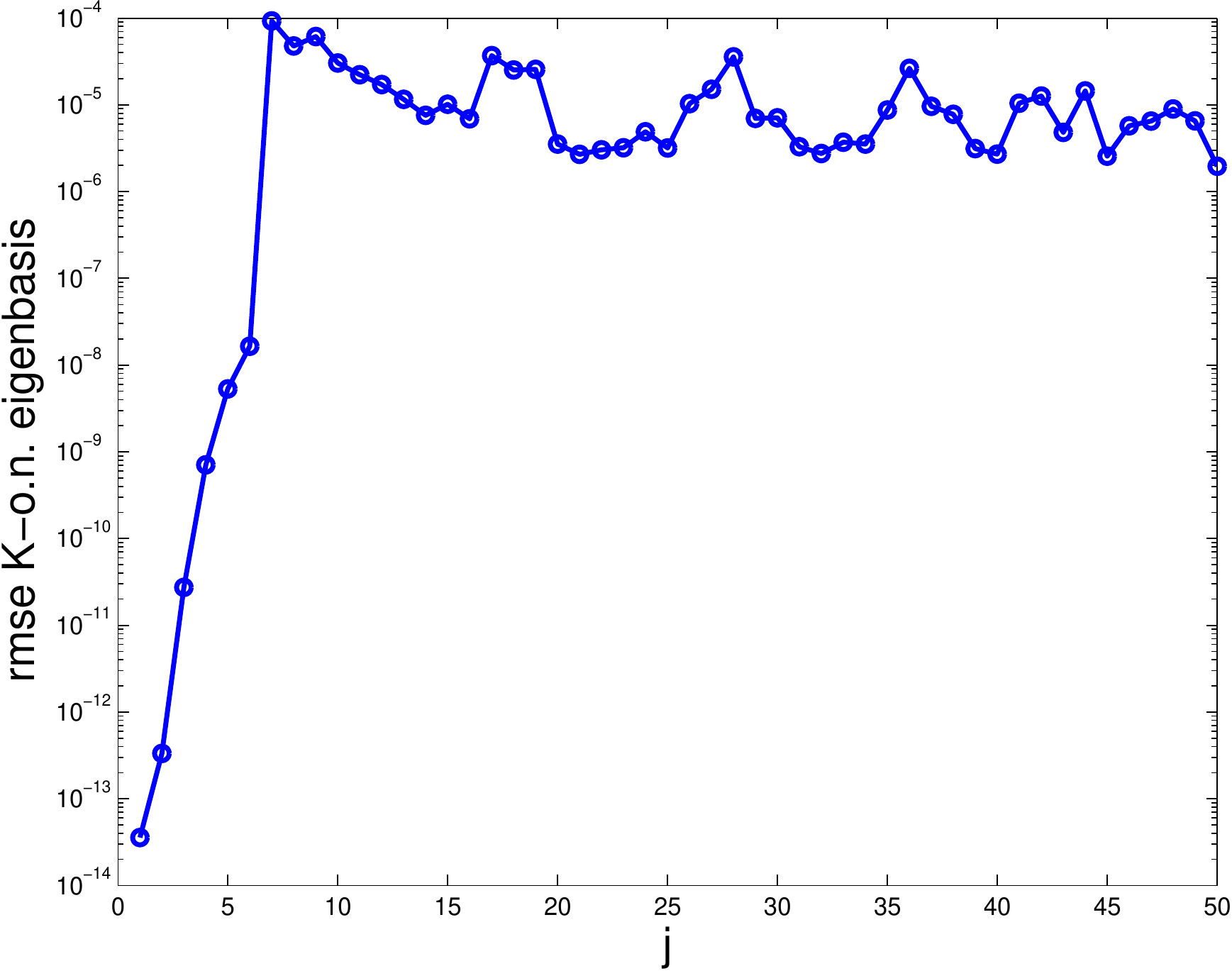}\\

\includegraphics[width=0.5 \textwidth]{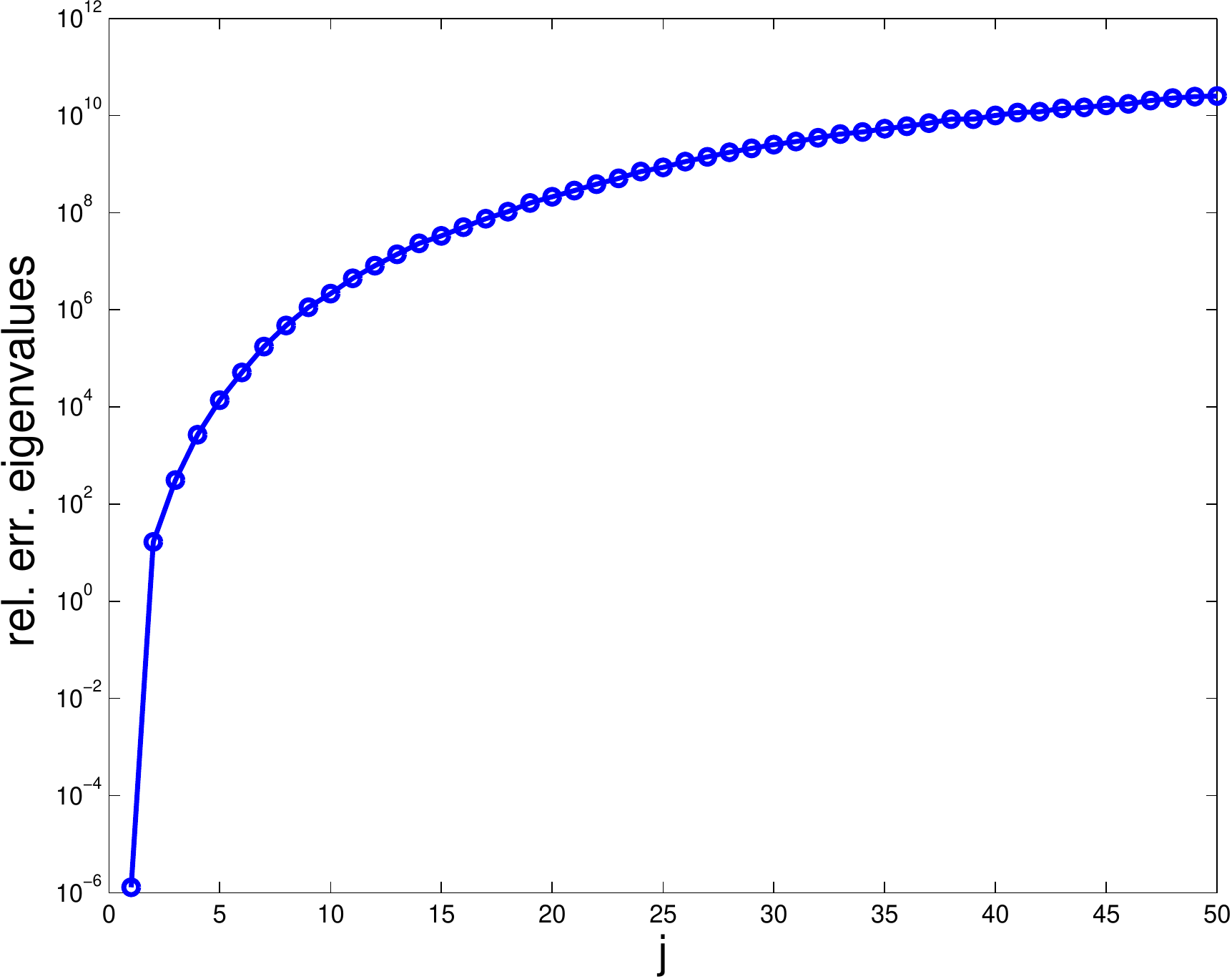}&
\includegraphics[width=0.5 \textwidth]{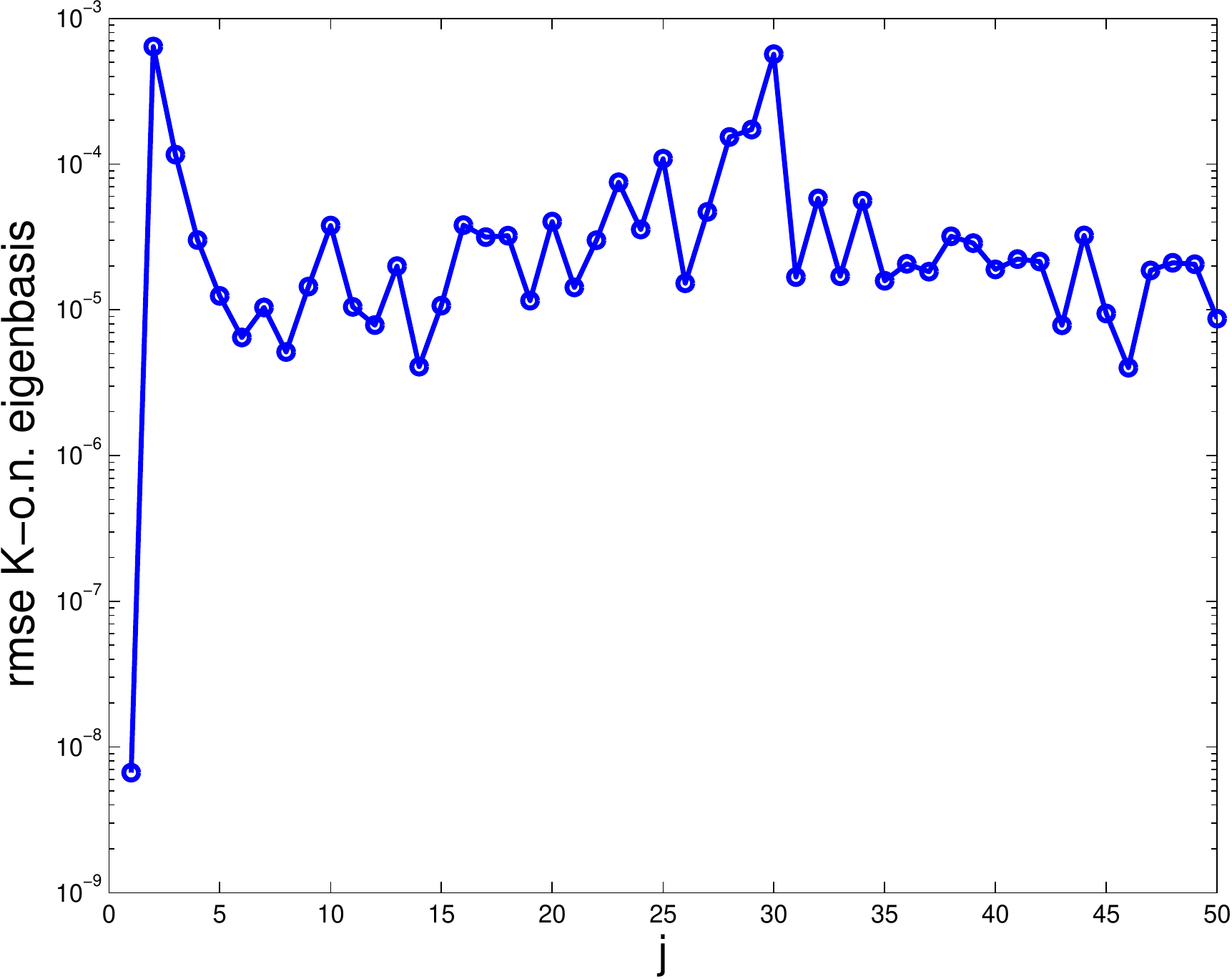}
\end{tabular}
\caption{Approximation of the eigenvalues (left) and the eigenbasis (right), 
for $\beta = 1,2,3$ (from top to bottom), $\varepsilon = 0$ 
and for $1\leqslant j\leqslant n$, $n = 50$, as discussed in Section \ref{sec:BBker}}\label{fig:approxEig}
\end{figure}

\bibliographystyle{abbrv}
\def\cprime{$'$}

\end{document}